\documentclass[11pt]{article}

\title{Optimal rates of convergence for persistence diagrams in Topological Data Analysis}
\author{Fr\'{e}d\'{e}ric Chazal\thanks{\url{frederic.chazal@inria.fr}}, Marc Glisse \thanks{\url{marc.glisse@inria.fr}}, Catherine Labru\`ere \thanks{\url{clabruer@u-bourgogne.fr}}, Bertrand Michel \thanks{\url{bertrand.michel@upmc.fr}}}
\date{\today}

\usepackage{amsmath,amsthm}
\usepackage{amssymb,latexsym}
\usepackage{eucal}

\usepackage{mathabx}

\usepackage{graphicx}
\usepackage[usenames,dvipsnames]{color}
\usepackage{pdfcolmk}
\usepackage[cmtip,arrow]{xy}
\usepackage{pb-diagram,pb-xy}

\usepackage{fullpage}
\usepackage{hyperref}

\usepackage{bbm}

\numberwithin{equation}{section}

\theoremstyle{plain}

\newtheorem{theorem}[equation]{Theorem}
\newtheorem{lemma}[equation]{Lemma}
\newtheorem{proposition}[equation]{Proposition}
\newtheorem{corollary}[equation]{Corollary}

\theoremstyle{definition}

\newtheorem{definition}[equation]{Definition}
\newtheorem*{definition*}{Definition}

\newtheorem*{example*}{Example}

\newtheorem*{notation*}{Notation}

\theoremstyle{remark}
\newtheorem*{remark}{Remark}

\newcommand{\pdfont}{\mathsf}
  
  \newcommand{\dgm}{\pdfont{dgm}}

\newcommand{\metricfont}{\mathrm}
  \newcommand{\bottle}{\metricfont{d_b}}

\newcommand{\Rips}{\operatorname{\mathrm{Rips}}}
\newcommand{\Cech}{\operatorname{\mathrm{Cech}}}

\newcommand{\Filt}{\operatorname{\mathrm{Filt}}}
\newcommand{\alphaC}{\operatorname{\mathrm{C}_\alpha}}

\newcommand{\e}{\varepsilon}
\newcommand{\one}{\mathbbm{1}}
\newcommand{\R}{\mathbb{R}}
\newcommand{\N}{\mathbb{N}}
\newcommand{\Z}{\mathbb{Z}}

\newcommand{\X}{\mathbb{X}}  
\newcommand{\bigM}{\mathbb{M}} 
\newcommand{\pk}{\operatorname{pk}}
\newcommand{\cov}{\operatorname{cv}}
\newcommand{\dgh}{\metricfont{d_{\textrm{\tiny GH}}}}
\newcommand{\dhaus}{\metricfont{d_{\textrm{\tiny H}}}}

\newcommand{\K}{\mathcal{C}}  

\newcommand{\hX}{\widehat \X}
\renewcommand{\P} {\mathbb{P}}
\newcommand{\E} {\mathbb{E}}
\newcommand{\TV}{\operatorname{TV}}

\begin{document}

\maketitle


\begin{abstract}
Computational topology  has recently known an important development toward data analysis, giving birth to the field of topological data analysis.
Topological persistence, or persistent homology, appears as a fundamental tool in this field. In this paper, we study topological persistence in
general metric spaces, with a statistical approach. We show that  the use of persistent homology can be naturally considered in general statistical
frameworks and persistence diagrams can be used as statistics with interesting convergence properties. Some numerical experiments are performed in
various contexts to illustrate our results.
\end{abstract}

\section{Introduction}

\paragraph{Motivations.}
During the last decades, the wide availability of measurement devices and simulation tools has led to an
explosion in the amount of available data in almost all domains of Science, industry, economy and even
everyday life. Often these data come as point clouds sampled in possibly high (or infinite) dimensional spaces. 
They are usually not uniformly distributed in
the embedding space but carry some geometric structure (manifold or more general
stratified space) which reflects important properties of the ``systems'' from which they have been generated.
Moreover, in many cases data are not embedded in Euclidean spaces and come as (finite) sets of points with pairwise distance information. This often happens, for example, with social network or sensor network data where each observation comes with a measure of its distance to the other observations: e.g., in a sensor network distributed in some domain, each sensor may not know its own position, but thanks to the strength of the signal received from the other sensors it may evaluate its distance from them. In such cases  data are just given as matrices of pairwise distances between the observations, i.e. as (discrete) metric spaces. Again, although they come as abstract spaces, these data often carry specific topological and geometric structures. 

A large amount of research has been done on dimensionality reduction, manifold learning
and geometric inference for data embedded in, possibly high dimensional, Euclidean spaces and assumed
to be concentrated around low dimensional manifolds; see for instance \cite{tenenbaum2000global,wang2012geometric,GenoveseEtAl2012} and the references
therein for recent results in this direction. However, the assumption of data lying on a manifold may fail in many applications. 
In addition, the strategy of representing data by points in Euclidean spaces may introduce large metric distortions as the data may lie in highly
curved spaces, instead of in flat Euclidean spaces, raising many difficulties in the analysis of metric data.
With the emergence of new geometric inference and algebraic topology tools, computational topology \cite{edelsbrunner2010computational} has recently
known an important development toward data analysis, giving birth to the field of Topological Data Analysis (TDA) \cite{c-td-09} whose aim is to infer
multiscale qualitative and quantitative relevant topological structures directly from the data. Topological persistence, more precisely {\em
persistent homology} appears as a fundamental tool for TDA. 
Roughly, {\it homology} (with coefficient in a field such as, e.g., $\mathbb{Z}/2\mathbb{Z}$) associates to any topological space $\mathbb{M}$, a
family of vector spaces (the so-called homology groups) $H_k(\mathbb{M})$, $k = 0, 1, \cdots$, each of them encoding $k$-dimensional features of
$\mathbb{M}$. The {\em $k^{th}$ Betti number} of $\mathbb{M}$, denoted $\beta_k$, is the dimension of $H_k(\mathbb{M})$ and measures the number of
$k$-dimensional features of $\mathbb{M}$: for example, $\beta_0$ is the number of connected components of $\mathbb{M}$, $\beta_1$ the number of
independent cycles or ``tunnels'', $\beta_2$ the number of ``voids'', etc...; see \cite{h-at-01} for a formal introduction to homology. 
Persistent homology provides a framework \cite{elz-tps-02, zc-cph-05,chazal2012structure} and efficient algorithms to encode the evolution of the
homology of families of nested topological spaces indexed by a set of real numbers 
that can often be seen as scale parameters, 
such as, e.g., the
sublevel sets of a function, union of growing balls, etc... The obtained multiscale topological information is then represented in a simple way as a
barcode or persistence diagram; see Figure \ref{fig:tore} and Section \ref{subsec:Simplcomp}.\\
In TDA, persistent homology has found applications in many fields, including neuroscience \cite{singh2008topological}, bioinformatics
\cite{kasson2007persistent}, shape classification \cite{ccgmo-ghsssp-09}, clustering \cite{cgos-pbcrm-2012} and  sensor networks
\cite{de2007homological}, to cite just a few.
It is usually computed for a {\it filtered simplicial complex} built on top of the available data, i.e. a nested family of simplicial complexes whose
vertex set is the data set (see Section \ref{subsec:Simplcomp}). The obtained persistence diagrams are then used as ``topological signatures'' to
exhibit the topological structure underlying the data; see Figure \ref{fig:PersistenceTDA}. The space of persistence diagrams is endowed with a
metric, the so-called {\em bottleneck distance}, that allows to compare the obtained signatures and thus to compare the topological structure underlying different data sets.
The relevance of this approach relies on stability results ensuring that close data sets, with respect to the so-called Gromov or Gromov-Hausdorff
distance, have close persistence diagrams \cite{steiner05stable,ccggo-ppmtd-09,chazal2012structure,cso-psgc-12}. However these results mainly remain
deterministic and thus often restrict to heuristic or exploratory uses in data analysis. 

The goal of this paper is to show that, thanks to recent results \cite{chazal2012structure,cso-psgc-12}, the use of persistent homology in TDA can be
naturally considered in general statistical frameworks and persistence diagrams can be used as statistics with interesting convergence properties. 

\begin{figure}
\centering
\includegraphics[width=0.9\columnwidth]{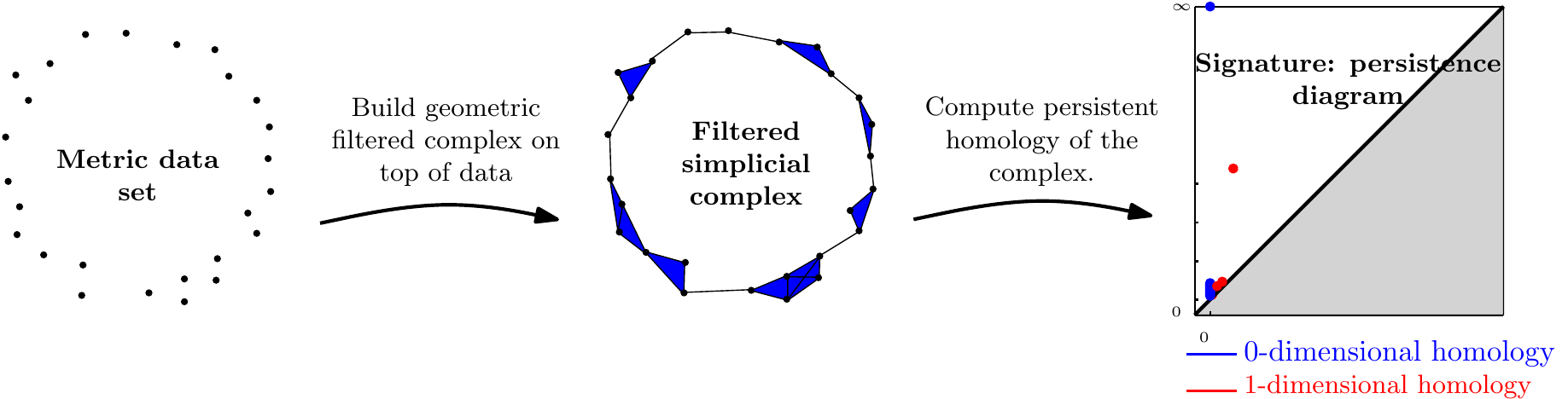}
\caption{A usual use of persistence in TDA.}
\label{fig:PersistenceTDA}
\end{figure}

\paragraph{Contribution.}
In this paper we assume that the available data is the realization of a probability distribution supported on an unknown compact metric space. We
consider the persistent homology of different filtered simplicial complexes built on top of the data and we study the rate of convergence of the
associated persistence diagrams to some well-defined persistence diagram associated to the support of the probability distribution, with a minimax
approach. 

More precisely, we assume that we observe a set of $n$ points $\hX_n = \{ X_1 \dots, X_n \}$ in a metric space $(\bigM, \rho)$, drawn i.i.d. from some
unknown measure $\mu$ whose support is a compact set denoted $\X_\mu \subseteq \bigM$. We then consider the persistent homology of the filtered
simplicial complexes $\Filt(\X_\mu)$ and $\Filt(\hX)$ built on top of $\X_\mu$ and $\hX_n$ respectively and we establish convergence rates of the
bottleneck distance between their persistence diagrams, $\bottle\left( \dgm (\Filt(\X_\mu)), \dgm (\Filt(\hX)) \right)$; see Section
\ref{sec:PersDiagInfMS} for explicit results.
It is important to notice that $\dgm (\Filt(\X_\mu))$ encodes topological properties of the support $\mu$ but not of the measure itself. As a
consequence, to obtain explicit convergence rates we assume that $\mu$ satisfies the so-called {\it $(a,b)$-standard assumption} for some constants
$a,b >0$: for any $x \in \X_\mu$ and any 
$r >0$, $\mu(B(x,r)) \geq \min(ar^b, 1)$. The following theorem illustrates the kind of results we obtain under such assumption.

{\bf Theorem (\ref{prop:lowerboundAbs} in Section \ref{sec:PersDiagInfMS}):}
{\em
Let $(\bigM,\rho)$, $a >0$ and $b  >0$ as above. Then for any measure $\mu$ satisfying the $(a,b)$-standard assumption
\begin{equation*}
\E \left[ \bottle  (   \dgm (\Filt(\X_\mu)) ,   \dgm(\Filt(\widehat \X_n)) ) \right]  \leq C  \left( \frac{ \ln n}  n  \right)^{1/b} 
\end{equation*}
where the constant $C$ only depends on  $a$ and $b$ (not on $\bigM$). 
Assume moreover that there exists a non isolated point $x$ in $\bigM$ and consider any sequence $(x_n) \in \left( \bigM  \setminus \{x\} \right) ^ {\mathbb N}$
such that $ \rho(x,x_n) \leq (a n) ^{-1/b} $. Then for any estimator $\widehat{\dgm}_n$ of
$\dgm (\Filt(\X_\mu))$:
$$
\liminf _{ n \rightarrow \infty}  \rho(x,x_n) ^{-1}  \E \left[ \bottle  (   \dgm (\Filt(\X_\mu)) , 
\widehat{\dgm}_n ) \right] \geq C '
$$
where $C'$ is an absolute constant.
}


Our approach relies on the general theory of persistence modules and our results follow from two recently proven properties of persistence diagrams \cite{cso-psgc-12, ccggo-ppmtd-09, chazal2012structure}.\\
First, as $\X_\mu$ can be any compact metric space (possibly infinite), the filtered complex $\Filt(\X_\mu)$ is usually not finite or even countable
and the existence of its persistence diagram cannot be established from the ``classical'' persistence theory \cite{zc-cph-05, elz-tps-02}. In our
setting, the existence of $\dgm (\Filt(\X_\mu))$ follows from the general persistence framework introduced in \cite{ccggo-ppmtd-09,
chazal2012structure}. Notice that although this framework is rather abstract and theoretical it does not have any practical drawback as only
persistence diagrams of complexes built on top of finite data are computed.\\
Second, a fundamental property of the persistence diagrams we are considering is their stability proven in \cite{cso-psgc-12}: the bottleneck distance
between $\dgm (\Filt(\X_\mu))$ and $\dgm(\Filt(\widehat \X_n))$ is upper bounded by twice the Gromov-Hausdorff distance between $\X_\mu$ and $\widehat
\X_n$. This result establishes a strong connection between our persistence estimation problem and support estimation problems. Upper bounds on the
rate of convergence of persistence diagrams are then easily obtained using the same arguments as the ones usually used to obtain convergence results
for support estimation with respect to the Hausdorff metric. We take advantage of this general remark to find rates of convergence of persistence
diagrams in general metric spaces (Section \ref{sec:PersDiagInfMS}) and also in the more classical case where the measure is supported in $\R^d$
(Section \ref{sec:Rd}). Using Le Cam's lemma, we also compute the corresponding lower bounds to check that the rates of convergence are optimal in the
minimax sense.

\paragraph{Related works.}
Although it is attracting more and more interest, the use of persistent homology in data analysis remains widely heuristic. There are relatively few
papers establishing connections between persistence and statistics and, despite a few promising results, the statistical analysis of homology,
persistent homology and more general topological and geometric features of data is still in its infancy. 

One of the first statistical results about persistent homology has been given in a parametric setting, by Bubenik  and Kim in
\cite{bubenik2007statistical}. They show for instance that for data sampled on an hypersphere according to a von-Mise Fisher distribution (among
other distributions), the Betti numbers can be estimated with the parametric rate $n ^{-1/2}$. However assuming that both the support and the
parametric family of the distribution are known are strong assumptions which are hardly met in practice. 

Closely related to our approach, statistical analysis of homology and of persistent homology has also been proposed very recently in  \cite{balakrishnan12minimax} and \cite{balakrishnan2013statistical} in the specific context of manifolds, i.e. when the geometric structure underlying the data is assumed to be a smooth submanifold of a Euclidean space. In the first paper, the authors exhibit minimax rates of convergence for the estimation of the Betti numbers of the underlying manifold under different models of noise. This approach is also strongly connected to manifold estimation results obtained in \cite{GenoveseEtAl2012}.
Our results are in the same spirit as \cite{balakrishnan12minimax} but extend to persistent homology and allow to deal with general compact metric
spaces. In the second paper, the authors develop several methods to find confidence sets for persistence diagrams using subsampling methods and kernel
estimators among other approaches. Although they tackle a different problem, it has some connections with the problem considered in the present paper
that we briefly mention in Section \ref{sec:confidence-set}. 

Both \cite{balakrishnan2013statistical} and our work start from the observation that persistence diagram inference is strongly connected to the better
known problem of support estimation. As far as we know, only few results about support estimation in  general metric spaces have been given in
the past. 
An interesting framework is proposed in \cite{Lesets}: in this paper the support estimation problem is tackled using kernel methods. On the other hand, a large
amount of literature is available for measure support estimation in $\R^d$; see for instance the review in \cite{Cuevas09} for more details. Note
that many results on this topic are given with respect to the volume of symmetric set difference - see for instance \cite{BiauCadreMasonPelletier09} and
references therein - while in our topological estimation setting we need convergence results for support estimation in Hausdorff metric.

The estimator $\widehat \X _n = \{X_1,\dots X_n\}$ and the Devroye and Wise estimator \cite{DevroyeWise80}, $ \hat S_n  = \bigcup_{i = 1} ^n \bar B (X_i,\varepsilon_n) $, where $\bar B(x,\varepsilon)$ denotes the closed ball centered at $x$ with radius $\varepsilon$, are both natural estimators of the support. The use of $\hat S_n$ is particularly relevant when the convergence of the measure of the symmetric set difference is considered but does not provide better results than $\widehat \X _n$ in our Hausdorff distance setting. 
The convergence rate of $\widehat \X _n$ to the support of the measure with respect to the Hausdorff distance is given in \cite{CuevasRCasal04} in $\R^d$. 
Support estimation in $\R^d$ has also been studied under various additional assumptions such as, e.g., convexity assumptions 
\cite{DumbgenWalther96,RodriguezCasal2007,CuevasEtal12} or through boundary fragments estimation \cite{KorostelevTsybakov93,
KorostelevSimarTsybakov95} just to name a few. 
Another classical assumption is that the measure has a density with respect to the Lebesgue measure. In this context, plug-in methods based on non
parametric estimators of the density have been proposed by \cite{CuevasFraiman97} and \cite{Tsybakov97}. We consider persistence diagram estimation in
the density framework of
\cite{SinghScottNowak09} in Section \ref{sec:Rd} and show in this particular context that $\widehat \X _n$ allows us to define a persistence
diagram estimator that reaches optimal rates of convergence in the minimax sense. 

A few other different methods have also been proposed for topology estimation in non deterministic frameworks such as the ones based upon
deconvolution approaches \cite{ccdm-dwmgi-11,Niyogi-Smale-Weinberger11}. Several recent attempts have also been made, with completely different
approaches, to study statistical persistence diagrams from a statistical point of view, such as \cite{mmj-pmspd-11} that studies probability measures
on the space of persistence diagrams or \cite{Bubenik12} that introduces a functional representation of persistence diagrams, the so-called
persistence landscapes, allowing to define means and variance of persistence diagrams. Notice that our results should easily extend to persistence
landscapes. 

\vskip15pt
The paper is organized as follows. Background notions and results on metric spaces, filtered simplicial complexes, and persistent homology  that are
necessary to follow the paper are presented in Section \ref{sec:preliminaries}. The rates of convergence for the estimation of persistence diagrams in
general metric spaces are established in Section \ref{sec:PersDiagInfMS}. We also study these convergence rates in $\R^d$ for a few classical problems
in Section \ref{sec:Rd}. Some numerical experiments illustrating our results are given in Section \ref{sec:PersDiagLearn}. All the technical proofs
are given in Appendix.

\section{Background}
\label{sec:preliminaries}


\subsection{Measured metric spaces}

Recall that a metric space is a pair $( \bigM,\rho)$ where $\bigM$ is a set and $\rho : \bigM \times \bigM \to \R$ is a nonnegative map such that for any $x,y,z \in \bigM$, $\rho(x,y) = 0$ if and only if $x = y$, $\rho(x,y) = \rho(y,x)$ and $\rho(x,z) \leq \rho(x,y)+ \rho(y,z)$. 
We denote by $\mathcal K(\bigM)$ the set of all the compact subsets of $\bigM$.
For a point $x\in\bigM$ and a subset $C\in \mathcal K(\bigM)$, the distance $d(x,C)$ of $x$ to $C$ is the minimum over all $y\in C$ of $d(x,y)$.
The Hausdorff distance $\dhaus(C_1,C_2)$ between two subsets $C_1, C_2 \in \mathcal K(\bigM)$ is the maximum over all points in $C_1$ of their distance to $C_2$ and over all points in $C_2$ of their distance to $C_1$ :
$$\dhaus(C_1,C_2) = \max\{\,\sup_{x \in C_1} d(x,C_2),\, \sup_{y \in C_2} d(y,C_1)\,\}  .  $$
Note that $(\mathcal K(\bigM), \dhaus)$ is a metric space and can be endowed with its Borel $\sigma$-algebra.

Two compact metric spaces $( \bigM_1,\rho_1)$ and $( \bigM_2,\rho_2)$ are {\em isometric} if there exists a bijection $\Phi : \bigM_1 \to \bigM_2$ that preserves distances, namely: $\forall x,y \in \bigM_1$, $\rho_2(\Phi(x),\Phi(y)) = \rho_1(x,y)$. Such a map $\Phi$ is called an {\em isometry}.  One way to compare two metric spaces is to measure how far these two metric spaces are from being isometric. The corresponding distance is called the {\em Gromov-Hausdorff distance} (see for instance \cite{burago2001course}). Intuitively, it is the infimum of their Hausdorff distance over all possible isometric embeddings of these two spaces into a common metric space.
\begin{definition} Let  $( \bigM_1,\rho_1)$ and $( \bigM_2,\rho_2)$ be two compact metric spaces. The Gromov-Hausdorff distance $\dgh\left( ( \bigM_1,\rho_1) \, , \, ( \bigM_2,\rho_2) \right)$ is the infimum of the real numbers $r \geq 0$ such that there exist a metric space $( \bigM,\rho)$ and subspaces $C_1$ and $C_2$ in $\mathcal K(\bigM)$ which are isometric to
$\bigM_1 $  and $\bigM_2 $ respectively and such that $\dhaus(C_1,C_2) < r$.
The Gromov-Hausdorff distance $\dgh$ defines a metric on the space $\mathcal K$ of isometry classes of compact metric spaces (see Theorem 7.3.30 in \cite{burago2001course}). 
\end{definition}
Notice that when $\bigM_1$ and $\bigM_2$ are subspaces of a same metric space $(\bigM, \rho)$ then $\dgh(\bigM_1,\bigM_2) \leq \dhaus(\bigM_1,\bigM_2)$.



\paragraph{Measure.}
Let $\mu$ be a probability measure on $(\bigM,\rho)$ equipped with its Borel algebra.  Let $\X_\mu$
denote the support of the measure $\mu$, namely the smallest closed set with probability one. In the following of the paper, we will assume that
$\X_\mu$ is compact and thus $\X_\mu \in \mathcal K (\bigM)$. Also note that $(\X_\mu, \rho)  \in \mathcal{K}$.

The main assumption we will need in the following of the paper provides a lower bound on the measure $\mu$. We say that $\mu$ satisfies the {\it standard assumption} if there exist $a'>0$, $r_0>0$ and $b >0$  such that 
\begin{equation} \label{ref:SdtAssump1}
\forall x \in \X_\mu, \ \forall r \in (0,r_0), \  \mu(B(x,r)) \geq a' r^b 
\end{equation}
where $B(x,r)$ denotes the open ball of center $x$ and radius $r$ in $\bigM$. 
This assumption is popular in the literature about set estimation (see for instance \cite{Cuevas09}) but it has generally been considered with $b = d$  in $\R^d$. Since  $\X_\mu$ is compact,  reducing the constant $a'$ to a smaller constant $a$ if necessary, we easily check that assumption (\ref{ref:SdtAssump1}) is equivalent to 
\begin{equation} \label{ref:SdtAssump2}
\forall x \in \X_\mu, \ \forall r>0 , \  \mu(B(x,r)) \geq 1\wedge a r^b
\end{equation}
where $x \wedge y $ denotes the minimum between  $x$ and $y$. We then say that $\mu$ satisfies the $(a,b)$-{\it standard assumption}.

\subsection{Simplicial complexes on metric spaces}
\label{subsec:Simplcomp}

\begin{figure}
\centering
\includegraphics[height=3cm]{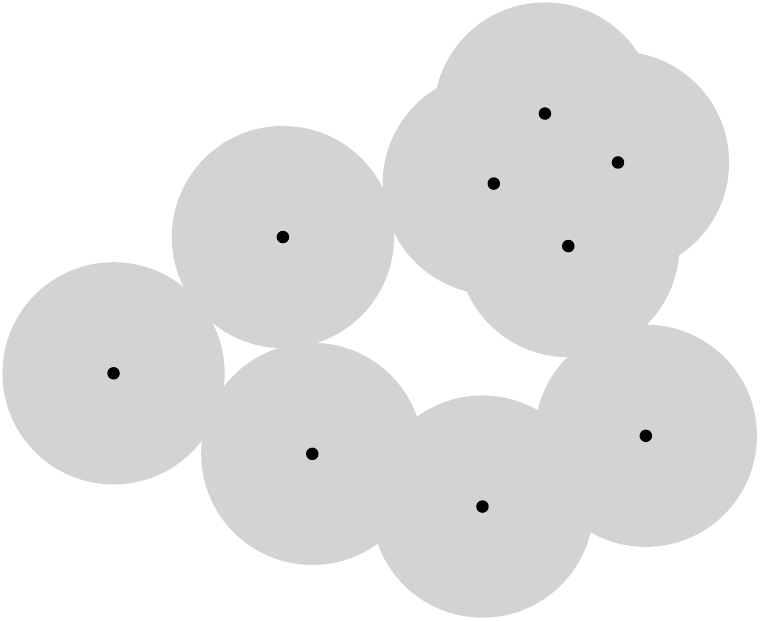}
\includegraphics[height=3cm]{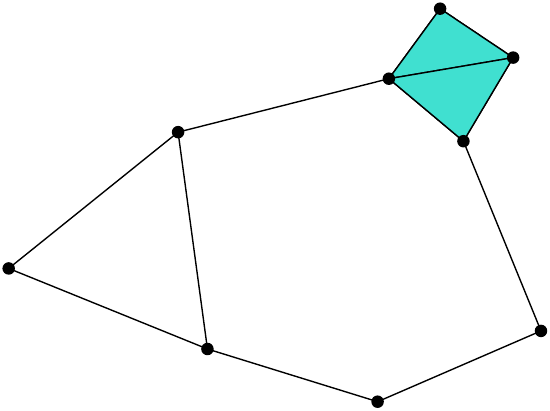}
\includegraphics[height=3cm]{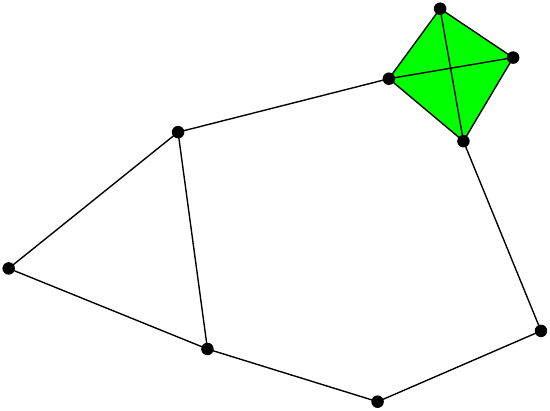}
\includegraphics[height=3cm]{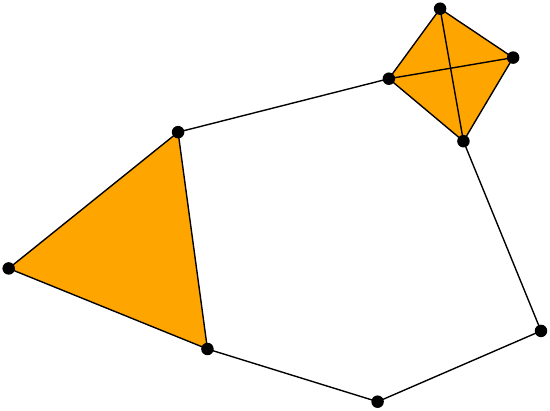}
\caption{From left to right: the $\alpha$ sublevelset of the distance function to a point
  set $\X$ in $\R^2$, the $\alpha$-complex, $\Cech_\alpha(\X)$ and $\Rips_{2\alpha}(\X)$.
The last two include a tetrahedron.}
\label{fig:complex}
\end{figure}
\paragraph{Geometric complexes.}
The geometric complexes we consider in this paper are built on top of metric spaces and come as nested families depending on a real parameter. Topological persistence is used to infer and encode the evolution of the topology of theses families as the parameter grows. 
For a complete definition of these geometric filtered complexes built on top of metric spaces and their use in TDA, we refer
to~\cite{cso-psgc-12}, Section 4.2. We only give here a brief reminder and
refer to Figure~\ref{fig:complex} for illustrations. A simplicial complex $\K$
is a set of simplexes (points, segments, triangles, etc) such that any face
from a simplex in $\K$ is also in $\K$ and the intersection of any two
simplices of $\K$ is a (possibly empty) face of these simplices. Notice that we do not assume such simplicial complexes to be finite. 
The complexes
we consider in this paper can be seen as a generalization of neighborhood graphs in dimension larger than $1$.

Given a metric space $\X$ which will also serve as the vertex set, the
\emph{Vietoris-Rips complex} $\Rips_\alpha(\X)$ is the set of simplices
$[x_0,\ldots,x_k]$ such that $d_\X(x_i,x_j)\leq\alpha$ for all $(i,j)$.
The \emph{\v Cech complex} $\Cech_\alpha(\X)$ is similarly defined as
the set of simplices $[x_0,\ldots,x_k]$ such that the $k+1$ closed balls
$B(x_i,\alpha)$ have a non-empty intersection. 
Note that these two complexes are related by
$\Rips_\alpha(\X)\subseteq\Cech_\alpha(\X)\subseteq\Rips_{2\alpha}(\X)$. Note also that these two families of complexes only depend on the pairwise distances between the points of $\X$.


When $\X$ is embedded in some larger metric space $\bigM$, we can extend
the definition of the \v Cech complex to
the set of simplices $[x_0,\ldots,x_k]$ such that the $k+1$ closed balls
$B(x_i,\alpha)$ have a non-empty intersection in $\bigM$ (not just in
$\X$). We can also define the \emph{alpha-complex} or
\emph{$\alpha$-complex} as
the set of simplices $[x_0,\ldots,x_k]$ such that, for some
$\beta\leq\alpha$ that depends on the simplex, the $k+1$ closed balls
$B(x_i,\beta)$ and the complement
of all the other balls $B(x,\beta)$ for $x\in\X$
have a non-empty intersection in $\bigM$.
In the particular case where
$\bigM=\R^d$, those two complexes have the same homotopy type (they are
equivalent for our purposes) as the union of the balls
$B(x,\alpha)$ for $x\in\X$,
as in Figure~\ref{fig:complex}, and the $\alpha$-complex only contains simplices
of dimension at most $d$.
Note that the union of the balls $B(x,\alpha)$ is also the $\alpha$-sublevel set of the distance to $\X$ function $d(.,\X)$, and as a consequence, those filtrations thus provide a convenient way to study the evolution of the topology of union of growing balls or sublevel sets of $d(.,\X)$ (see Figure \ref{fig:complex} and Section \ref{sec:PersDiagLearn} for more examples).

There are several other families that we could also have considered, most
notably witness complexes \cite{cso-psgc-12}. Extending our results to them is
straightforward and yields very similar results, so we will restrict to the
families defined above in the rest of the paper.

All these families of complexes
have the fundamental property that they are non-decreasing with
$\alpha$; for any $\alpha\leq\beta$, there is an inclusion of
$\Rips_\alpha(\X)$ in $\Rips_\beta(\X)$, and similarly for the \v Cech,
and Alpha complexes. They are thus called \emph{filtrations}.
%
%
In the following, the notation $\Filt(\X) := (\Filt_\alpha (\X))_{\alpha
\in \mathcal A}$  denotes one of the filtrations defined above.


\paragraph{Persistence diagrams.} An extensive presentation of
persistence diagrams is available in \cite{chazal2012structure}. We recall a few
definitions and results that are needed in this paper.

\begin{figure}
\centering\includegraphics[height=6cm]{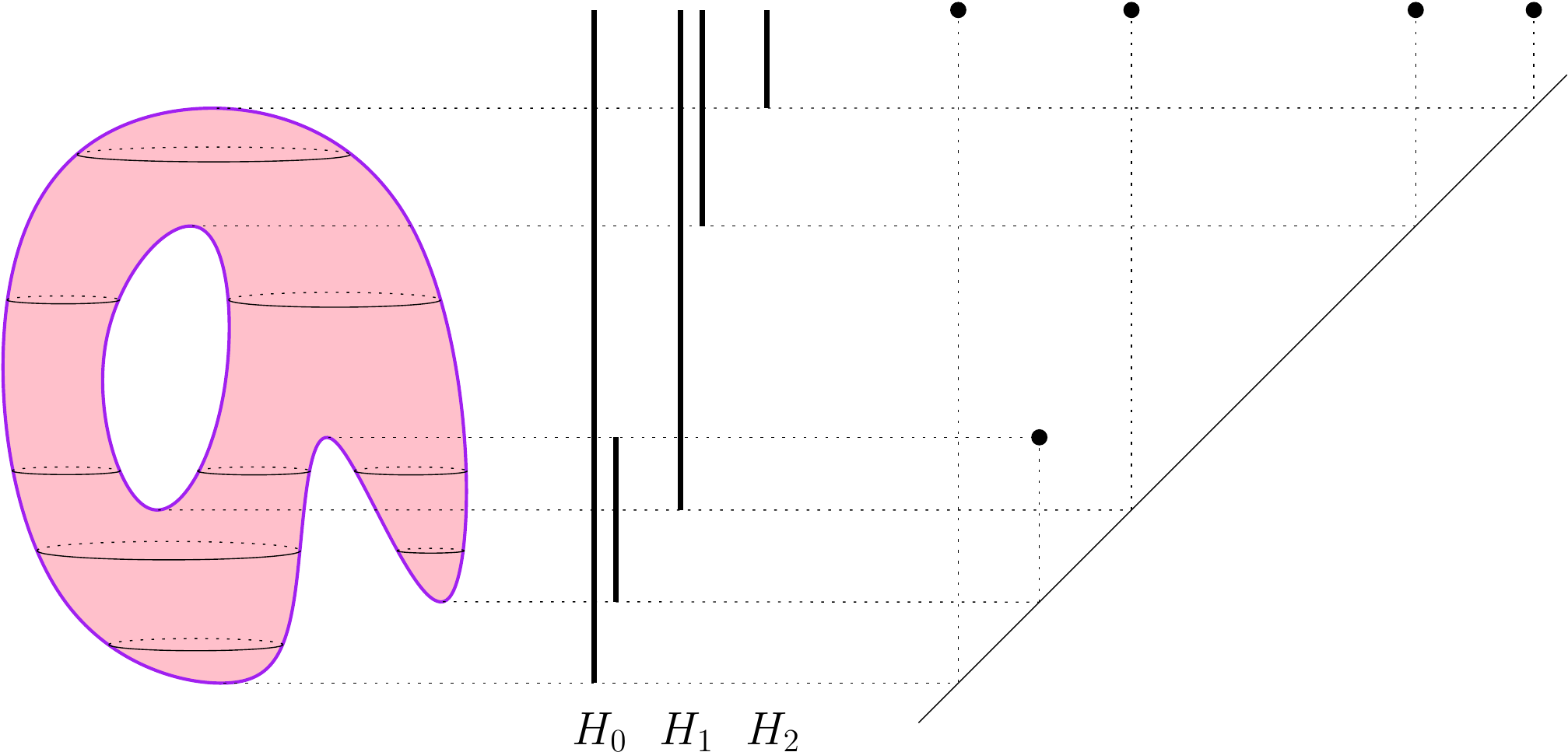}
\caption{A torus $\mathbb{T}$ filtered by its $z$-coordinate:
$\Filt_\alpha=\{P\in\mathbb{T}|P_z\leq\alpha\}$, its persistence
barcode, and its persistence diagram.}
\label{fig:tore}
\end{figure}

\begin{figure}
\centering\includegraphics[height=5cm]{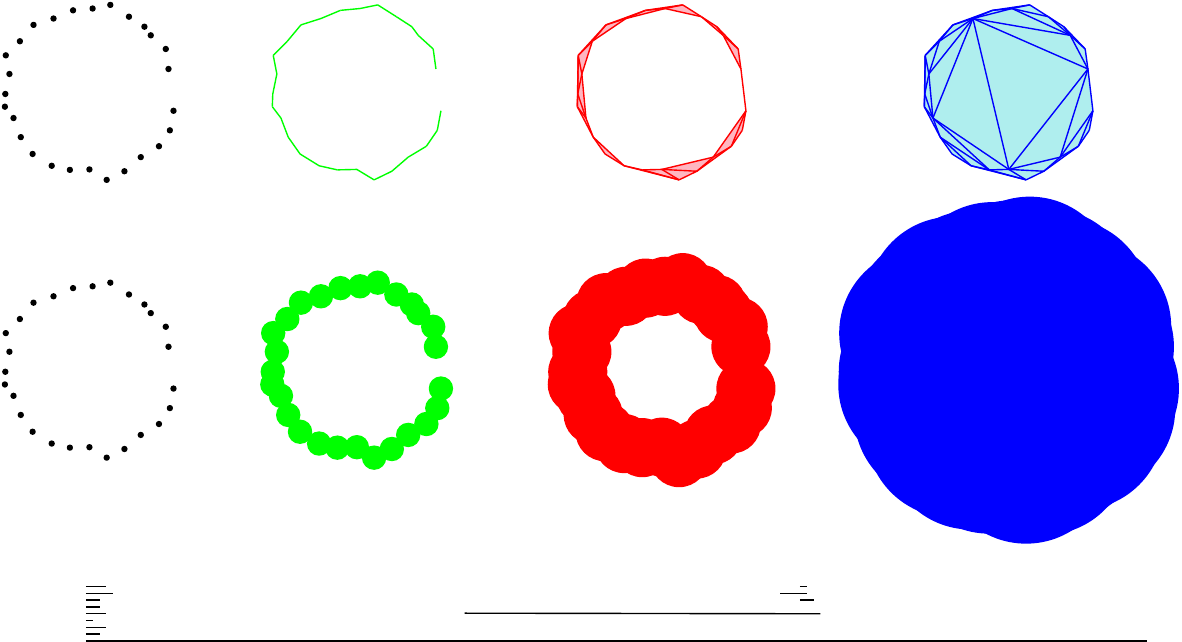}
\caption{An $\alpha$-complex filtration, the sublevelset filtration of
the distance function, and their common persistence barcode (they are
homotopy equivalent).}
\label{fig:alpha-persist}
\end{figure}

We first give the intuition behind persistence.
Given a filtration as above, the topology of $\Filt_\alpha (\X)$ changes as $\alpha$ increases: new connected components can appear, existing connected components can merge, cycles and cavities can appear and can be filled, etc.
Persistent homology is a tool that tracks these changes, identifies \emph{features} and associates a \emph{lifetime} to them. For instance, a connected component is a feature that is born at the smallest $\alpha$ such that the component is present in $\Filt_\alpha (\X)$, and dies when it merges with an older connected component. Intuitively, the longer a feature persists, the more relevant it is.

We now formalize the presentation a bit.
Given a filtration as above, we can apply the $\Z_2$-homology functor 
\footnote{The notion of (simplicial) homology is a classical concept in algebraic topology that provides powerful tools to formalize and handle the notion of topological features of a simplicial complex in an algebraic way. For example the $0$-dimensional homology group $H_0$ represents the $0$-dimensional features, i.e. the connected components of the complex, $H_1$ represents the $1$-dimensional features (cycles), $H_2$ represents the $2$-dimensional features (cavities),... See, e.g. \cite{h-at-01} for an introduction to simplicial homology.}
and get a sequence of vector spaces $(H(\Filt_\alpha (\X)))_{\alpha \in
\mathcal A}$, where the inclusions $\Filt_\alpha (\X)\subseteq
\Filt_\beta (\X)$ induce linear maps $H(\Filt_\alpha (\X))\rightarrow
H(\Filt_\beta (\X))$. In many cases, this sequence can be decomposed as
a direct sum of intervals, where an interval is a sequence of the form
\[0\rightarrow \ldots\rightarrow 0\rightarrow \Z_2\rightarrow
\ldots\rightarrow \Z_2\rightarrow 0\rightarrow \ldots\rightarrow 0\]
(the linear maps $\Z_2\rightarrow\Z_2$ are all the identity). These
intervals can be interpreted as features of the (filtered) complex, such
as a connected component or a loop, that appear at parameter
$\alpha_\textrm{birth}$ in the filtration and disappear at parameter
$\alpha_\textrm{death}$. An interval is determined uniquely by these
two parameters. It can be represented as a segment whose extremities
have abscissae $\alpha_\textrm{birth}$ and $\alpha_\textrm{death}$; the
set of these segments is called the barcode of $\Filt(\X)$. An interval
can also be represented as a point in the plane,
where the $x$-coordinate indicates the birth time and the $y$-coordinate
the death time. The set of points (with multiplicity) representing the
intervals is called the persistence diagram $\dgm(\Filt(\X))$. Note that
the diagram is entirely contained in the half-plane above the diagonal
$\Delta$ defined by $y=x$, since death always occurs after birth.
\cite{chazal2012structure} shows that this diagram is still well defined even in
cases where the sequence might not be decomposable as a finite sum of
intervals, and in particular $\dgm(\Filt(\X))$ is well defined for any
compact metric space $\X$ \cite{cso-psgc-12}.
Note that for technical reasons, the points
of the diagonal $\Delta$ are considered as part of every persistence
diagram, with infinite multiplicity.
The most persistent features (supposedly the most important) are those represented by the longest bars in the barcode, i.e. the points furthest from the diagonal in the diagram, whereas points close to the diagonal can be interpreted as noise.

\begin{figure}
\centering\includegraphics[height=5cm]{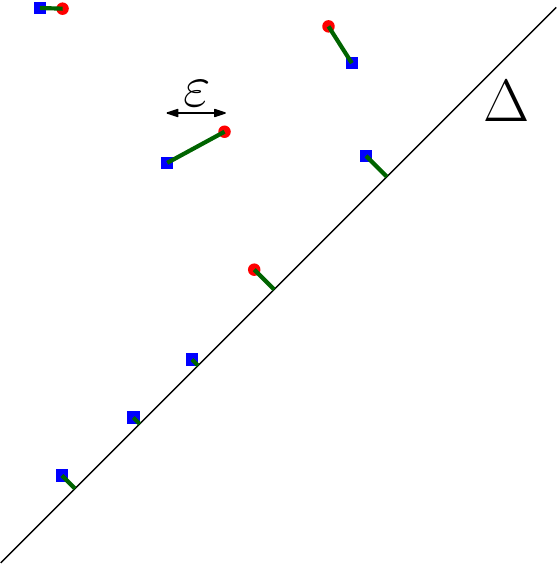}
\caption{Two diagrams at bottleneck distance $\varepsilon$.}
\label{fig:bottleneck}
\end{figure}

The space of persistence diagrams is endowed with a metric called the
\emph{bottleneck distance} $\bottle$. Given two persistence diagrams, it is
defined as the infimum, over all perfect matchings of their points, of the
largest $L^\infty$-distance between two matched points, see Figure \ref{fig:bottleneck}.
The presence of the diagonal in all diagrams means we can consider partial matchings of the off-diagonal points, and the remaining points are matched to the diagonal.
With more details, given two diagrams $\dgm_1$ and $\dgm_2$, we can define a
matching $m$ as a subset of
$\dgm_1\times\dgm_2$ such that every point of $\dgm_1\setminus\Delta$ and $\dgm_2\setminus\Delta$ appears exactly once in $m$. The bottleneck distance is then:
\[\bottle(\dgm_1,\dgm_2)=\inf_{\textrm{matching }m}\ \max_{(p,q)\in m}\ ||q-p||_\infty.\]
Note that points close to the diagonal $\Delta$ are easily matched to the diagonal, which fits with their interpretation as irrelevant noise.


A fundamental property of persistence diagrams, proved in
\cite{chazal2012structure}, is their \emph{stability}. If $\X$ and $\tilde \X$
are two compact metric spaces then one has
\begin{equation} \label{bot-dgh}
\bottle\left( 
\dgm (\Filt(\X)), 
\dgm (\Filt(\tilde \X))
\right)
\leq
2 \dgh \left(\X, \tilde \X\right).
\end{equation}
Moreover, if $\X$ and $\tilde \X$ are embedded in the same metric space
$(\bigM,\rho)$ then one has
\begin{equation} \label{bot-dgh_dhaus}
\bottle\left( 
\dgm (\Filt(\X)), 
\dgm (\Filt(\tilde \X))
\right)
\leq
2 \dgh \left(\X, \tilde \X\right)
\leq  2 \dhaus \left(\X, \tilde \X\right).
\end{equation}

Notice that these properties are only metric properties: they do not involve here any probability measure on $\X$ and $\tilde \X$.

\section{Persistence diagrams estimation in metric spaces}
\label{sec:PersDiagInfMS}
Let $( \bigM,\rho)$ be a metric space. Assume that we observe $n$
points $X_1 \dots, X_n$ in $\bigM$ drawn i.i.d. from some unknown measure $\mu$ whose support is a compact set denoted $\X_\mu$.

\subsection{From support estimation to persistence diagram estimation} 
\label{subsec:supest2PersEst}

The Gromov-Hausdorff distance allows to compare $\X_\mu$ with compact metric spaces 
not necessarily embedded in  $\bigM$. We thus consider $(\X_\mu,\rho) $ as an element of $\mathcal K$ (rather than an element of $\mathcal K (\bigM)$). In the
following, an \emph{estimator} $\hX$ of $\X_\mu$ is thus a function of $X_1 \dots, X_n$ which takes its values in $\mathcal K$ and which is measurable for the Borel algebra induced by $\dgh$.

Let $\Filt(\X_\mu)$ and $\Filt(\hX)$ be two filtrations defined on $\X_\mu$ and $\hX$.  The statistical analysis of persistence diagrams proposed in
the sequel starts from the following key fact: according to
(\ref{bot-dgh}), for any $\varepsilon >0$:
\begin{equation} \label{eq:dgmdgh}
 \P \left( 
  \bottle\left( 
\dgm (\Filt(\X_\mu)), \dgm (\Filt(\hX))
\right) > \varepsilon 
\right) 
\leq  \P \left( \dgh(  \X_\mu, \hX ) >  2 \varepsilon \right)  
\end{equation}
where the probability  corresponds to the  product measure $\mu^{\otimes n}$. Our strategy then consists in finding an estimator of the support which is close
for the $\dgh$ distance. Note that this general strategy of estimating $\X_\mu$ in $\mathcal K$ is not only of theoretical interest. Indeed as mentioned in the introduction, in some cases the space $\bigM$ is unknown and the observations $X_1 \dots, X_n$ are just known through their matrix of pairwise distances $\rho(X_i,X_j)$, $i,j = 1, \cdots, n$. The use of the Gromov-Hausdorff distance then allows to consider this set of observations as an abstract metric space of cardinality $n$ without taking care of the way it is embedded in $\bigM$.  

This general framework embraces the more standard approach consisting in estimating the support by restraining the values of  $\hX$ to
$\mathcal K (\bigM)$. According to (\ref{bot-dgh_dhaus}), in this case, for any $\varepsilon >0$:
\begin{equation} \label{eq:dbdh}
 \P \left(  
 \bottle \left( 
\dgm (\Filt(\X_\mu)), 
\dgm (\Filt(\hX))
\right) > \varepsilon \right)
 \leq  
 \P \left( \dhaus(  \X_\mu, \hX ) >  2 \varepsilon \right).
\end{equation}
Thanks to equations (\ref{eq:dgmdgh}) and (\ref{eq:dbdh}) the problem of persistence diagrams estimation boils down to the better known problem of estimating the support of a measure.


Let $\widehat \X _n := \{ X_1, \dots ,X_n\}$ be a set of independent observations sampled according to $\mu$ endowed with the restriction of the distance $\rho$. This finite metric space is a natural estimator of the support $\X_\mu$.
In several contexts discussed in the following,
$\widehat \X _n$ shows optimal rates of convergence for the estimation of $\X_\mu$ with respect to the Hausdorff and Gromov-Hausdorff distance. From (\ref{eq:dbdh}) we will then obtain upper bounds on the rate of convergence of $\Filt(\hX_n)$, and we will need to find the corresponding lower bounds to prove the optimality for topological inference issue.

In the next subsection, we tackle persistence diagram estimation in the general framework of abstract metric spaces. We will consider more particular contexts later in the paper.

\subsection{Convergence of persistence diagrams} 
Cuevas and Rodr\'iguez-Casal give in \cite{CuevasRCasal04}
the rate of convergence in Hausdorff distance of  $\widehat \X _n$  for some probability measure $\mu$ satisfying an $(a,d)$-standard assumption on
$\R^d$. In this section, we consider the more general  context where  $\mu$ is a  probability measure satisfying an $(a,b)$-standard assumption on a
metric space $(\bigM,\rho)$, with $b>0$. We give below the rate of convergence of $\widehat \X _n $ in this context. The
proof follows the lines of the proof of Theorem 3 in \cite{CuevasRCasal04}.

\begin{theorem} \label{theo:ConvHaus} Assume that a probability  measure $\mu$  on $\bigM$ satisfies the  $(a,b)$-standard assumption. Then,
for any $\varepsilon >0$:
$$ \P \left( \dhaus(  \X_\mu, \hX_n ) >  2  \varepsilon \right) 
 \leq \frac{2^b}{a \e^b} \exp(-na \e^b) \wedge 1 . $$
Moreover, there exist two constants $C_1$ and $C_2$ only depending on $a$ and $b$ such that 
$$ \limsup_{n \rightarrow \infty} \; \left( \frac n {\log n}  \right) ^{1/b}  \dhaus(  \X_\mu, \hX_n )  \leq C_1  \hskip 0.5cm \textrm{ almost
surely},$$
and 
$$  \lim_{n \rightarrow \infty} \; \P \left(   \dhaus(  \X_\mu, \hX_n ) \leq C_2 \left( \frac  {\log n} n \right) ^{1/b} \right)  = 1  . $$
\end{theorem}

Since $\dgh (  \X_\mu, \hX_n ) \leq \dhaus(  \X_\mu, \hX_n )$ the above theorem also holds when the Gromov distance is replaced by the Gromov-Hausdorff distance. In practice this allows to consider $\hX_n$ as an abstract metric space without taking care of the way it is embedded in the, possibly unknown, metric space $\bigM$. 

Using (\ref{eq:dgmdgh}) and (\ref{bot-dgh_dhaus}), we then derive from the previous result the following corollary for the convergence rate of the persistence diagram $\Filt(\widehat \X _n)$ toward $\Filt(\X_\mu)$.

\begin{corollary} \label{cor:UpperRatesAbstract}
Assume that the probability  measure $\mu$  on $\bigM$ satisfies the  $(a,b)$-standard assumption,  then for any $\varepsilon >0$:
\begin{equation} \label{ubP} 
 \P \left(  
 \bottle \left( 
\dgm (\Filt(\X_\mu)), 
\dgm (\Filt(\widehat \X _n))
\right) > \e  \right)
 \leq \frac{2^b}{a \e^b} \exp(-na \e^b) \wedge 1 .
 \end{equation} 
 Moreover, 
 $$ \limsup_{n \rightarrow \infty} \; \left( \frac n {\log n}  \right) ^{1/b} \bottle \left( 
\dgm (\Filt(\X_\mu)), 
\dgm (\Filt(\widehat \X _n))
\right)  \leq C_1  \hskip 0.5cm \textrm{ almost
surely},$$
and 
$$  \lim_{n \rightarrow \infty} \; \P \left(  \bottle \left( 
\dgm (\Filt(\X_\mu)), 
\dgm (\Filt(\widehat \X _n))
\right) \leq C_2 \left( \frac  {\log n} n \right) ^{1/b} \right)  = 1  . $$
where  $C_1$ and $C_2$ are the same constants as in Theorem \ref{theo:ConvHaus}.
\end{corollary}


\subsection{Optimal rate of convergence}

Let $\mathcal  P (a,b,\bigM)$  be the set of all the
probability measures on the metric space $(\bigM,\rho)$ satisfying the $(a,b)$-standard assumption on $\bigM$. 
 \begin{equation*} \label{Pad}
\mathcal P  (a,b,\bigM) :=  \left\{ \mu  \textrm{ on } \bigM \: | \:  \X_\mu \textrm{ is compact and } \forall x \in \X_\mu,  \, \forall r > 0 , \, \mu \left(B(x,r) \right) \geq  1 \wedge a r^b  \right\}  .
\end{equation*}
The next theorem gives upper and lower bounds for the rate of convergence of persistence diagrams. The upper bound comes as a consequence of Corollary \ref{cor:UpperRatesAbstract}, while the lower bound is established using the so-called Le Cam's lemma (see Lemma \ref{Lem:Lecam} in Appendix). 


\begin{theorem} \label{prop:lowerboundAbs}
Let $(\bigM,\rho)$ be a metric space and let $ a >0$ and $ b  >0 $. Then:
\begin{equation} \label{ubE}
 \sup_{\mu \in \mathcal P (a,b,\bigM)} \E \left[ \bottle  (   \dgm (\Filt(\X_\mu)) ,   \dgm(\Filt(\widehat \X_n)) ) \right]  \leq C  \left( \frac{ \ln n}  n  \right)^{1/b} 
\end{equation}
where the constant $C$ only depends on  $a$ and $b$ (not on $\bigM$). 
Assume moreover that there exists a non isolated point $x$ in $\bigM$ and consider any sequence $(x_n) \in \left( \bigM  \setminus \{x\} \right) ^ {\mathbb N}$
such that $ \rho(x,x_n) \leq (a n) ^{-1/b} $. Then for any estimator $\widehat{\dgm}_n$ of
$\dgm (\Filt(\X_\mu))$:
$$
\liminf _{ n \rightarrow \infty}  \rho(x,x_n) ^{-1}   \sup_{\mu \in \mathcal P (a,b,\bigM)}\,  \E \left[ \bottle  (   \dgm (\Filt(\X_\mu)) , 
\widehat{\dgm}_n ) \right] \geq C '
$$
where $C'$ is an absolute constant.
\end{theorem}

Consequently, the estimator $ \dgm(\Filt(\widehat \X_n)) $ is minimax optimal on the space $\mathcal P  (a,b,\bigM)$ up to a logarithmic term as soon
as we can find a non-isolated point in $\bigM$ and a sequence $(x_n)$ in $\bigM$ such that $ \rho(x_n,x) \sim  (a n) ^{-1/b}$.
%
This is obviously the case for the Euclidean space $\R^d$.

\subsection{Confidence sets for persistence diagrams} 
\label{sec:confidence-set}

Corollary~\ref{cor:UpperRatesAbstract} can also be used to find confidence sets for persistence diagrams. Assume that $a$ and $b$ are known and let $\Psi : \eta \rightarrow  \exp(-\eta) / \eta$. Then for $\alpha \in (0,1)$,
\[
B_{d_b}\left( \dgm (\Filt(\X_\mu)), \left[ \frac 1 {n a} \Psi^{-1} \left(\frac  {\alpha}{n 2 ^b} \right) \right]^{1/b} \right)
\] is a confidence region for $\dgm \left(\Rips(\mu(K) \right))$ of level $1-\alpha$. Nevertheless, in practice the coefficients $a$ and $b$ can be unknown. In $\R^d$, the coefficient $b$ can be taken equal to the ambient dimension $d$ in many situations. Finding lower bounds on the coefficient $a$  is a tricky problem that is out of the scope of the paper. 
Alternative solutions have been proposed recently in \cite{balakrishnan2013statistical} and we refer the reader to this paper for more details.


\section{Persistence diagram estimation in $\boldsymbol{\R^d}$}
\label{sec:Rd}

In this section, we study the convergence rates of persistence diagram estimators for data embedded in $\R^d$. In particular we study  two situations
of interest proposed respectively in \cite{SinghScottNowak09} and \cite{GenoveseEtAl2012} in the context of measure support estimation. In the first situation the measure has a density with respect to the Lebesgue measure on $\R^d$ whose behavior is controlled near the boundary of its support. In the second case, the measure is supported on a manifold. These two frameworks are complementary and provide realistic frameworks for topological inference in $\R^d$.


\subsection{Optimal persistence diagram estimation for nonsingular measures on $\R^d$} 
\label{sub:OptSmoothBoundaries}

Paper \cite{SinghScottNowak09} is a significant breakthrough for level set estimation through density estimation. It presents a fully
data-driven procedure, in the spirit of Lepski's method, that is adaptive to unknown local density regularity
and achieves
a Hausdorff error control
that is minimax optimal
for a class of level sets with very general shapes. In particular, the assumptions of \cite{SinghScottNowak09} describe the smoothness
of the density near the boundary of the support. 

In this section, we propose to study persistence diagram inference in the framework of \cite{SinghScottNowak09} since this framework is very intuitive and natural. Nevertheless, we do not use the estimator of \cite{SinghScottNowak09}  for this task since we only consider here the support estimation problem (and not the more general level set issue as in \cite{SinghScottNowak09}). Indeed, we will see that the estimator $\hat \X_n$ has the optimal rate of convergence for estimating the support according to $\dhaus$, as well as for estimating the persistence diagram. We now recall the framework of \cite[Section 4.3]{SinghScottNowak09} corresponding to support set estimation.

Let $X_1, \dots,X_n$ be i.i.d.  observations drawn from an unknown probability measure $\mu$ having density $f$ with respect to the Lebesgue measure and defined on
a compact set $\chi\subset\R^d$. Let $\X_{f}$ denote the support of $\mu$, 
and let $G_0 := \{ x \in \chi \: : \: f(x) > 0\}$.  The boundary of a set $G$ is denoted  $\partial G$ and for any
$\varepsilon >0$, $I_\varepsilon(G) := \bigcup_{x \; | \;  B(x,\varepsilon ) \subset G} B(x,\varepsilon ) $  is the $\varepsilon$-inner of $G$.
 The two main assumptions of \cite{SinghScottNowak09}  are the following:
\begin{description}
\item $[A]$ : the density $f$ is upper bounded by $f_{\textrm{max}} >0 $ and there exist constants $\alpha$, $C_a$, $\delta_a >0$ such that for all $x \in  G_0$
with $f(x) \leq \delta_a$, 
$  f(x) \geq C_a  \, d(x, \partial G_0 )^\alpha  .$ 
\item $[B]$ :  there exist constants $\varepsilon_0 > 0$ and $C_b > 0$ such that for all $\varepsilon \leq \varepsilon_0$, $I_\varepsilon(G_0) \neq
\emptyset$ and $d(x, I_\varepsilon(G_0)  )  \leq  C_b \,  \varepsilon $ for all $x \in \partial G_0$.
\end{description}
We denote by $\mathcal F(\alpha)$ the set composed of all the densities on $\chi$ satisfying assumptions $[A]$ 
and $[B]$, for a fixed set of positive constants $C_a$, $C_b$, $\delta_a$, $\varepsilon_0$, $f_{\textrm{max}}$, $p$ and $\alpha$.

Assumption $[A]$ describes how fast the density increases in the neighborhood of the boundary of the support: the smaller $\alpha$, the easier the
support may be possible to detect. Assumption $[B]$ prevents the boundary from having arbitrarily small features (as for cusps). We refer to
\cite{SinghScottNowak09}  for more details and discussions about these two assumptions and their connections with assumptions in other works.

For persistence diagram estimation, we are interested in estimating the support $\X_{f}$ whereas the assumptions  $[A]$  and $[B]$  involve the set
$G_0$. However, as stated in Lemma~\ref{lem:dhG0front} (given in Appendix~\ref{proofSec41}), these two sets are here almost identical in the
sense that $ \dhaus(G_0, \X_{f}) = 0$. Moreover, it can be proved that under assumptions $[A]$ and $[B]$, the measure $\mu$ also satisfies the
standard assumption with $b = \alpha + d$ (see  Lemma~\ref{lem:dhG0front}).
According to  Proposition~\ref{prop:lowerboundAbs}, the estimator $\dgm(\Filt(\widehat \X_n))$ thus converges in
expectation towards $\dgm(\Filt( \X_f))$ with a rate upper bounded by $(\log n  / n) ^{1/(d+\alpha )}$. We also show that this rate is minimax over
the sets $\mathcal F (\alpha)$  by adapting the ideas of the proof  given in \cite{SinghScottNowak09} for the Hausdorff lower bound.
\begin{proposition} \label{prop:dplusalpha} 
\begin{enumerate}
\item For all $n\geq 1$,
\begin{equation*}
sup_{ f \in \mathcal F (\alpha)}  \E \left[ \bottle  (   \dgm (\Filt(\X_f)) , 
\dgm(\Filt(\widehat \X_n)) \right] \leq  C   \left( \frac n {\log n} \right) ^{-1/(d + \alpha)}
\end{equation*}
 where $C$ is a constant depending only on $C_a$, $C_b$, $\delta_a$, $\varepsilon_0$, $f_{\textrm{max}}$, $p$ and $\alpha$.
 \item There exists $c >0$ such that
\begin{equation*}
\inf_{\widehat{\dgm}_n}  \sup_{ f \in \mathcal F (\alpha)}  \E \left[ \bottle  (   \dgm (\Filt(\X_f)) , 
\widehat{\dgm}_n ) \right] \geq c   n ^{-1/(d + \alpha)}
\end{equation*}
for $n$ large enough. The infimum is taken over all possible estimators $\widehat{\dgm}_n$ of $\dgm (\Filt(\X_f))$ based on $n$ observations.
\end{enumerate}
\end{proposition}


\begin{remark} Paper \cite{SinghScottNowak09} is more generally about adaptive level set estimation. For this problem, Singh {\it et al.} define an histogram based estimator. Let $\mathcal A_j$ denote the collection of cells, in a regular partition of $\chi =[0,1]^d$ into hypercubes of dyadic side length $2^{-j}$. Their estimator $\hat f$ is the histogram $ \hat f (A) = \hat P (A)/ \mu(A)$, where $\hat P (A) = \sum _{i=1 \dots n} 1_{X_i  \in A}$. For estimating the level set $ G_{\gamma} : = \{ x  |  f(x) \geq   \gamma \} $, they consider the estimator
 $$ \hat G_{\gamma,j} = \bigcup _{A \in \mathcal A_j \; | \;  \hat f (A) > \gamma } A  .$$
 It is proved in \cite{SinghScottNowak09} that $\hat G_{\gamma,\hat j}$ achieves optimal rates of convergence for estimating the level sets, with $\hat j$ chosen in a data driven way. Concerning support estimation, they also show that $\hat G_{0,j}$ achieves optimal rates of convergence for estimating $G_0$. We have seen that in this context it is also the case for the estimator $\X _n$. Since no knowledge of $\alpha$ is required for this last estimator,  we thus prefer to use this simpler estimator in this context.
\end{remark}

\subsection{Optimal rates of convergence of persistence diagram estimation for singular measures in $\R^D$}

In this subsection, we consider the estimation of the support of a singular measure embedded in $\R^D$. A classical assumption in this context is to
suppose that the support of the singular measure is a Riemannian manifold. As far as we know, rates of
convergence for manifold estimation, namely for the estimation of the support of a singular probability measure supported on a Riemannian manifold of
$\R ^D$,  have only been studied recently in \cite{GenoveseEtAl2012} and \cite{GPVW12}. These papers assume several noise models, which all could be
considered in this context of persistence diagram estimation. However, for the sake of simplicity, we only study here the problem where no
additional noise is observed, which is referred as the {\it noiseless model} in the first of these two papers. As before, upper bounds given in
\cite{GenoveseEtAl2012} on the rates of
convergence for the support estimation in Hausdorff distance directly provide upper bounds on the rates of convergence of the persistence diagram
of the support. Before giving the rates of convergence we first recall and discuss the assumptions of  \cite{GenoveseEtAl2012}. 

For any $r>0$ and any set $A \subset \R^d$, let $A \oplus \varepsilon := \bigcup_{a \in A } B(a, r)$.  Let $ \Delta(\X_\mu) $  be the largest $r$ such
that each point in $\X_\mu \oplus r $ has a unique projection onto $\X_\mu$, this quantity has been introduced by Federer in \cite{Federer59}, it is
called reach or condition number in the literature. 

For a fixed positive integer $d < D$, for some  fixed positive constants $b$, $B$, $\kappa$ and for a fixed compact domain $\chi$ in $\R^p$, \cite{GenoveseEtAl2012}
defines the set of probability measures $\mathcal H := \mathcal H (d,A,B,\kappa,\chi)$ on $\chi$ satisfying the two following assumptions:
\begin{itemize}
 \item $[H_1]$ The support of the measure $\mu$ is a compact Riemannian manifold $\X_\mu$ (included in $\chi$) of dimension $d$ whose reach  satisfies
 \begin{equation} \label{eq:reach}
 \Delta(\X_\mu) \geq \kappa .
 \end{equation}
 \item $[H_2]$ The measure $\mu$ is assumed to have a density $g$ with respect to $d$-dimensional volume measure 
 $vol_d$ on $\X_\mu$, such that
 \begin{equation} \label{eq:densityGenov}
  0 < A \leq \inf_{y \in \X_\mu } g(y) \leq \sup_{y \in \X_\mu } g(y) \leq B  < \infty .
 \end{equation}
\end{itemize}


These two assumptions can be easily connected to the standard assumption. Indeed, according to \cite{Niyogi-Smale-Weinberger08} and using $[H_1]$, for all  $r \leq  \kappa$ there exists some constant $C >0$ such that for any $x \in \X_\mu$, we have 
\begin{eqnarray*}
vol_d \left ( B(x , r) \cap \X_\mu \right) & \geq&  C \left( 1 - \frac {r^2} {4 \kappa^2}  \right) ^{d/2} r ^d  \\
& \geq&  C' r ^d 
\end{eqnarray*}
and the same holds for $\mu$ according to $[H_2]$. Thus, if we take $\hat \X_n$ for estimating the support $\X_\mu$ in this context, we then obtain the rate of convergence $(\frac{\log n}{n})^{ 1 /  d}$ according to Theorem \ref{theo:ConvHaus}.
Nevertheless, this rate is not minimax optimal on the spaces $\mathcal H$ as shown by Theorem 2 in \cite{GenoveseEtAl2012}. Indeed the
correct rate is $n^{- 2 / d}$. For proving this result, \cite{GenoveseEtAl2012} proposes some  ``theoretical" estimator that can not be computed in
practice. As far as we know, no usable and optimal estimator has been proposed in the literature for this issue. In consequence, the situation is the same for the estimation of persistence diagrams in this context. The following proposition shows that the optimal rates of convergence for support estimation are the same as for the persistence diagram estimation in this context.
\begin{proposition} \label{prop:GPVW}
Assume that we observe an $n$-sample under the previous assumptions, then there exist two  constants $C$  and $C'$ depending only on $\mathcal H$ such
that
\begin{equation}
 C    n ^{-2/d } \leq  \inf_{\widehat{\dgm}_n }   \sup_{ \mu  \in  \mathcal H} \E \left[ \bottle  (   \dgm (\Filt(\X_\mu)) , 
\widehat{\dgm}_n ) \right] \leq C'   n ^{-2/d }
\end{equation}
where the infimum is taken over all the estimators of the persistence diagram.
\end{proposition}
Of course this result is only of theoretical interest since it is not based on estimators which are usable in practice.

\section{Experiments}
\label{sec:PersDiagLearn}

A series of experiments were conducted in order to illustrate the behavior of the persistence diagrams under sampling of metric spaces endowed with a probability measure and to compare the convergence performance obtained in practice with the theoretical results obtained in the previous sections.

\paragraph{Spaces and data.} We consider four different metric spaces, denoted $\bigM_1$, $\bigM_2$, $\bigM_3$ and $\bigM_4$  hereafter, that are described below. 
\begin{itemize}
\item[{\bf $\bigM_1$}] {\bf (Lissajous curve in $\mathbb{R}^2$):} the planar curve with the parametric equations $x(t) = \sin (3t+ \pi/2)$, $y(t) = \sin(2t)$, $t \in [0, 2 \pi]$ (see Figure \ref{fig:spaces}, left). Its metric is the restriction of the Euclidean metric in $\mathbb{R}^2$ and it is endowed with the push forward by the parametrization of the uniform measure on the interval $[0,2 \pi]$. 
\item[{\bf $\bigM_2$}] {\bf (sphere in $\mathbb{R}^3$):} the unit sphere in $\mathbb{R}^3$ (see Figure \ref{fig:spaces}, center). Its metric is the restriction of the Euclidean metric in $\mathbb{R}^3$ and it is endowed with the uniform area measure on the sphere. 
\item[{\bf $\bigM_3$}] {\bf (torus in $\mathbb{R}^3$):} the torus of revolution in $\mathbb{R}^3$ with the parametric equations $x(u,v) = (5 + \cos(u)) \cos(v)$, $y(u,v) = (5 + \cos(u)) \sin(v)$ and $z(u,v) = \sin(u)$, $(u,v) \in [0, 2 \pi]^2$ (see Figure \ref{fig:spaces}, right). Its metric is the restriction of the Euclidean metric in $\mathbb{R}^3$ and it is endowed with the push forward by the parametrization of the uniform measure on the square $[0, 2 \pi]^2$.
\item[{\bf $\bigM_4$}] {\bf (rotating shape space):} for this space we used a 3D character from the SCAPE database \cite{scape} and  considered all the images of this character from a view rotating around it. We converted these images in gray color and resized these images to $300\times 400 = 120,000$ pixels (see Figure \ref{fig:spaceM4}). Each is then identified with a point in $\mathbb{R}^{120,000}$ where the $i^{th}$ coordinate is the level of gray of the $i^{th}$ pixel. Moreover, we normalized these images by projecting them on the unit sphere in  $\mathbb{R}^{120,000}$. The metric space $\bigM_4$ is the obtained subset of the unit sphere with the restriction of the Euclidean metric in $\mathbb{R}^{120,000}$. As it is parametrized by a circular set of views, it is endowed with the push forward of the uniform measure on the circle. 
\end{itemize}

\begin{figure}
\centering
\includegraphics[height=3.5cm]{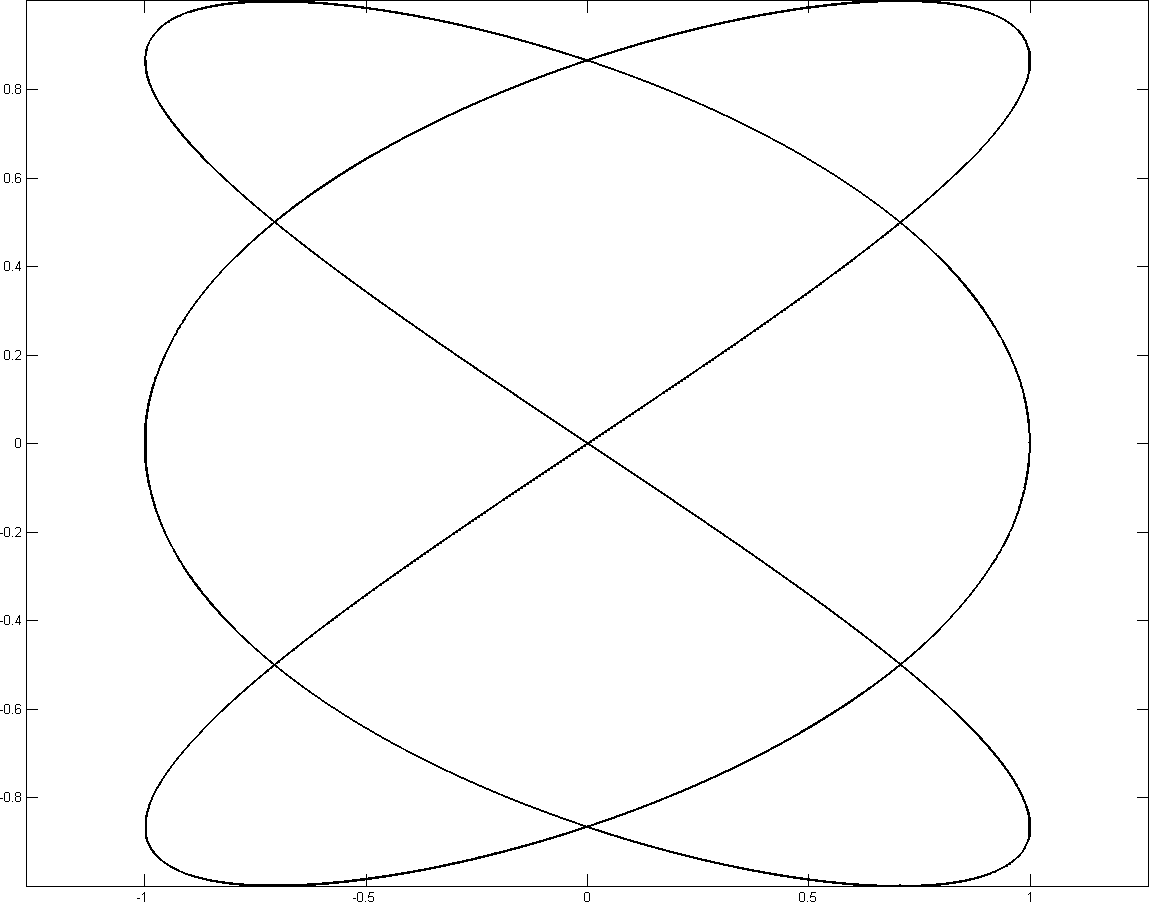}\qquad
\includegraphics[height=3.5cm]{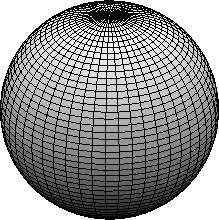}\qquad
\includegraphics[height=3.5cm]{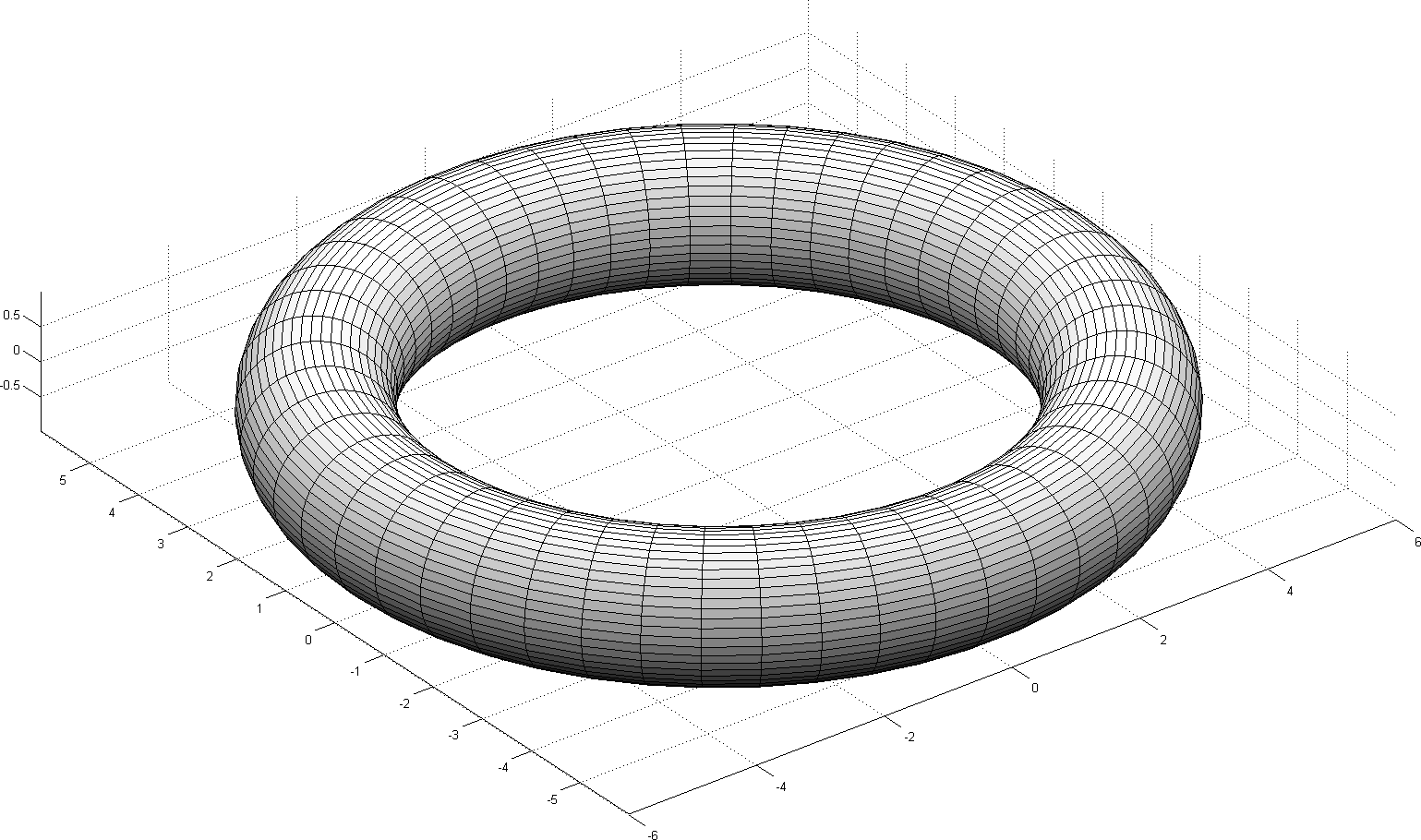}
\caption{The spaces $\bigM_1$, $\bigM_2$ and $\bigM_3$.}
\label{fig:spaces}
\end{figure}

\begin{figure}
\centering
\includegraphics[width = 0.15\columnwidth]{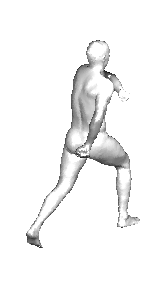}
\includegraphics[width = 0.15\columnwidth]{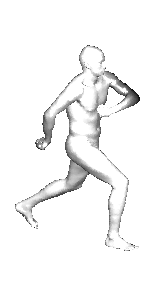}
\includegraphics[width = 0.17\columnwidth]{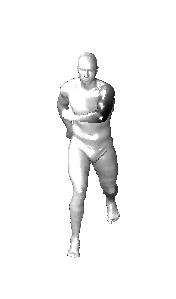}
\includegraphics[width = 0.17\columnwidth]{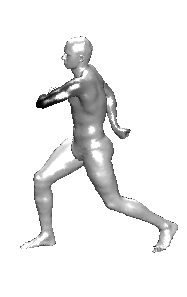}
\includegraphics[width = 0.15\columnwidth]{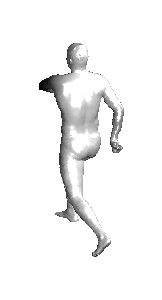}
\includegraphics[width = 0.17\columnwidth]{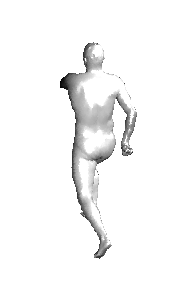}
\caption{Images sampled from the space $\bigM_4$.}
\label{fig:spaceM4}
\end{figure}


\paragraph{The experiments.} From each of the measured metric spaces $\bigM_1$, $\bigM_2$, $\bigM_3$ and $\bigM_4$ we sampled $k$ sets of $n$ points for different values of $n$ from which we computed persistence diagrams for different  geometric complexes (see Table \ref{table:spaces}). 
For $\bigM_1$, $\bigM_2$ and $\bigM_3$ we have computed the persistence diagrams for the $1$ or $2$-dimensional homology of the $\alpha$-complex built on top of the sampled sets. As $\alpha$-complexes have the same homotopy type as the corresponding union of balls, these persistence diagrams are the ones of the distance function to the sampled point set \cite{edelsbrunner1995union}. So, for each $n$ we computed the average bottleneck distance between the obtained diagrams and the  persistence diagram of the distance to the metric space from which the points were sampled. 
For $\bigM_4$, as it is embedded in a very high dimensional space, computing the $\alpha$-complex is practically out of reach. So we have computed the persistence diagrams for the $1$-dimensional homology of the Vietoris-Rips complex built on top of the sampled sets. The obtained results are described and discussed below.
\begin{itemize}
\item {\bf Results for $\bigM_1$:} we approximated the $1$-dimensional homology persistence diagram of the distance function to the Lissajous curve $\dgm(\bigM_1)$ by sampling $\bigM_1$ with $500,000$ points and computing the persistence diagram of the corresponding $\alpha$-complex. As the Hausdorff distance between our sample and $\bigM_1$  was of order $10^{-5}$ we obtained a sufficiently precise approximation of $\dgm(\bigM_1)$ for our purpose. The diagram $\dgm(\bigM_1)$ is represented in blue on the left of Figure \ref{fig:LissajousConvergence}. 
For each $n$, the average bottleneck distance between $\dgm(\bigM_1)$ and the persistence diagrams obtained for the $k=300$ randomly sampled sets $\X_n$ of size $n$ has been used as an estimate $\hat \E$ of $\E \left[ \bottle  (   \dgm (\alphaC(\bigM_1)) ,   \dgm(\alphaC(\widehat \X_n)) ) \right]$ where $\alphaC$ denotes the $\alpha$-complex filtration. $\log(\hat \E)$ is plotted as a function of $\log(\log(n)/n)$ on Figure \ref{fig:LissajousConvergence}, right. As expected, since the Lissajous curve is $1$-dimensional, the points are close to a line of slope $1$. 

\item{\bf Results for $\bigM_2$ and $\bigM_3$:} the persistence diagrams $\dgm(\bigM_2)$ and $\dgm(\bigM_2)$ of the distance functions to $\bigM_2$ and $\bigM_3$ are known exactly and are represented in blue on Figures \ref{fig:SphereConvergence} and \ref{fig:TorusConvergence}, left, respectively. Notice that we considered the $2$-dimensional homology for $\bigM_2$ and $1$-dimensional homology for $\bigM_3$.
For $i=2,3$ and for each $n$, the average bottleneck distance between $\dgm(\bigM_i)$ and the persistence diagrams obtained for the $k=100$ randomly sampled sets $\X_n$ of size $n$ has been used as an estimate $\hat \E$ of $\E \left[ \bottle  (   \dgm (\alphaC(\bigM_i)) ,   \dgm(\alphaC(\widehat \X_n)) ) \right]$ where $\alphaC$ denotes the $\alpha$-complex filtration. $\log(\hat \E)$ is plotted as a function of $\log(\log(n)/n)$ on Figures  \ref{fig:SphereConvergence} and \ref{fig:TorusConvergence}, right. As expected, since the sphere and the torus are $2$-dimensional, the points are close to a line of slope $1/2$. 

\item{\bf Results for $\bigM_4$:} As in that case we do not know the persistence diagram of the Vietoris-Rips filtration built on top of $\bigM_4$, we only computed the $1$-dimensional homology persistence diagrams of the Vietoris-Rips filtrations built on top of $20$ sets of $250$ points each, randomly sampled on $\bigM_4$. All these diagrams have been plotted on the same Figure \ref{fig:shaperotation}, left. The right of Figure \ref{fig:shaperotation}  represents a 2D embedding of one of the $250$ points sampled data set using the Multidimensional Scaling algorithm (MDS). Since $\bigM_4$ is a set of images taken according a rotating point of view, it carries a cycle structure. This structure is reflected in the persistence diagrams that all have one point which is clearly off the diagonal. Notice also a second point off the diagonal which is much closer to it and that probably corresponds to the pinching in $\bigM_4$ visible at the bottom left of the MDS projection. 
\end{itemize}

\begin{table}
\centering
\begin{tabular}{|l|c|c|c|}
\hline
\hline
{\bf Space} & {\bf $k$ (sampled sets for each $n$)} & {\bf $n$ range} & {\bf Geometric complex} \\
\hline
\hline
$\bigM_1$ & 300 & $[2100:100:3000]$ & $\alpha$-complex  \\
\hline
$\bigM_2$ & 100 & $[12000:1000:21000]$ & $\alpha$-complex  \\
\hline
$\bigM_3$ & 100 & $[4000:500:8500]$ & $\alpha$-complex  \\
\hline
$\bigM_4$ & 20 & 250 & Vietoris-Rips complex \\
\hline
\hline
\end{tabular}
\caption{Sampling parameters and geometric complexes where $[n_1: h: n_2]$ denotes the set of integers $\{n_1, n_1+h, n_1+2h, \cdots n_2 \}$.}
\label{table:spaces}
\end{table}

\begin{figure}
\centering
\includegraphics[width = 0.45\columnwidth]{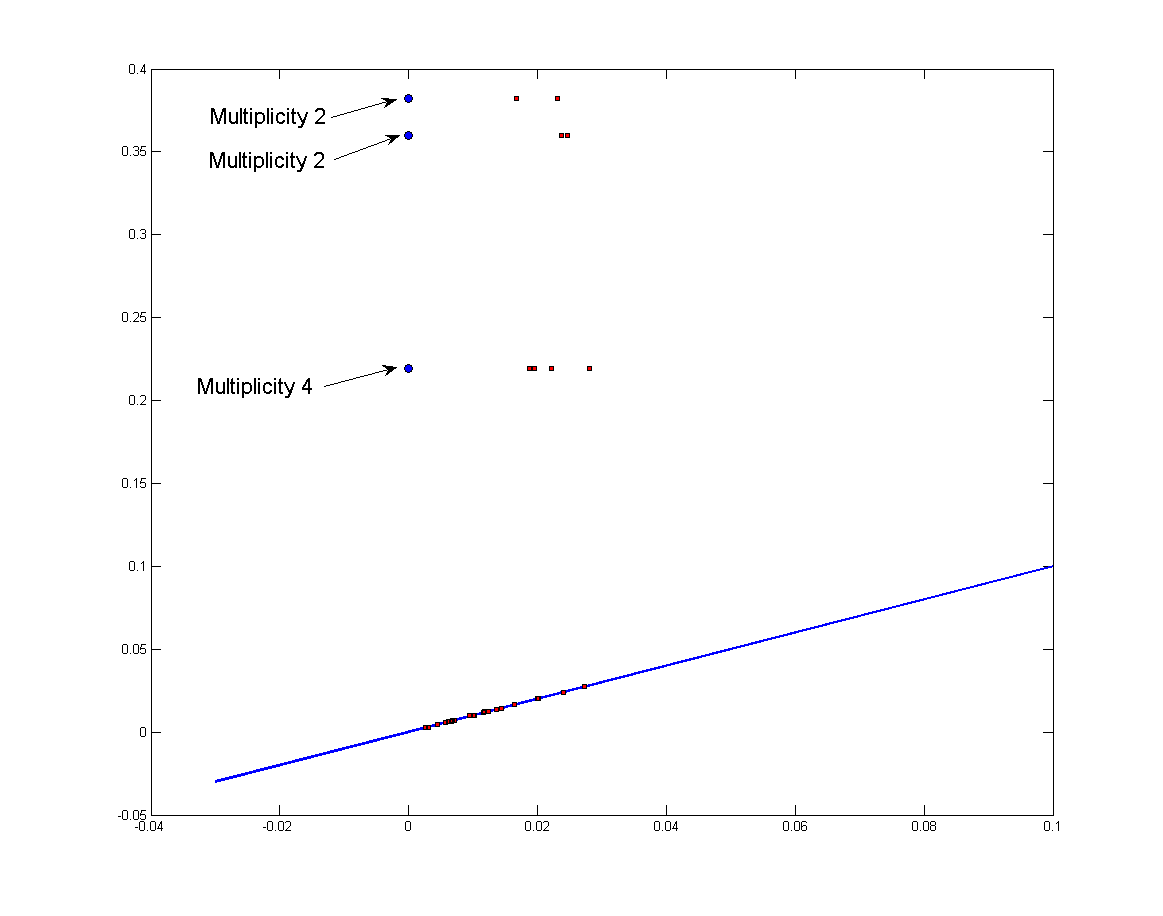}
\includegraphics[width = 0.45\columnwidth]{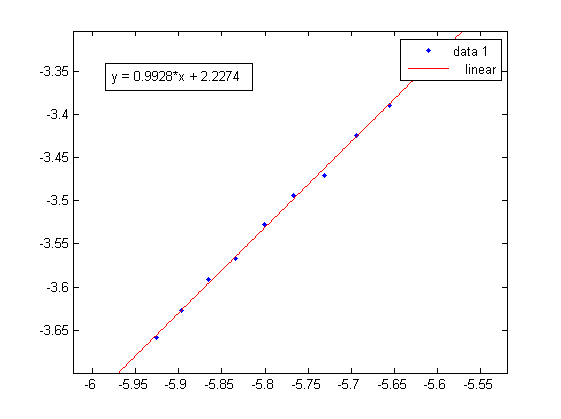}
\caption{Convergence rate for the persistence diagram of the $\alpha$-filtration built on top of points sampled on $\bigM_1$. Left: in blue the persistence diagram $\dgm(\bigM_1)$ of the distance to $\bigM_1$ ($1$-dimensional homology); in red a persistence diagram of the $\alpha$-filtration built on top of $n = 2100$ points randomly sampled on $\bigM_1$. Right: the $x$-axis is $\log(\log(n)/n)$ where $n$ is the number of points sampled on $\bigM_1$. The $y$-axis is the $\log$ of the estimated expectation of the bottleneck distance between the diagram obtained from an $\alpha$-filtration built on top of $n$ points sampled on $\bigM_1$ and $\dgm(\bigM_1)$.}
\label{fig:LissajousConvergence}
\end{figure}

\begin{figure}
\centering
\includegraphics[width = 0.45\columnwidth]{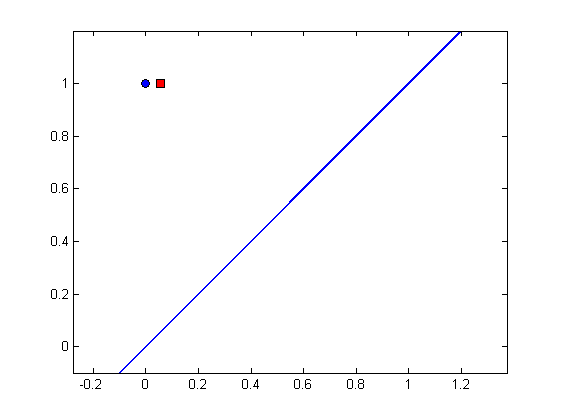}
\includegraphics[width = 0.45\columnwidth]{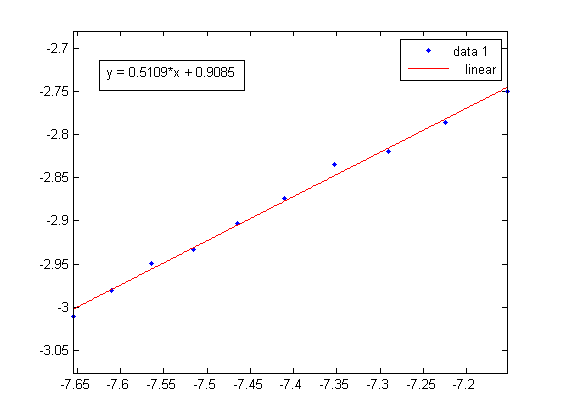}
\caption{Convergence rate for the persistence diagram of the $\alpha$-filtration built on top of points sampled on $\bigM_2$. Left: in blue the persistence diagram $\dgm(\bigM_2)$ of the distance to $\bigM_2$ ($2$-dimensional homology); in red a persistence diagram of the $\alpha$-filtration built on top of $n = 12000$ points randomly sampled on $\bigM_2$. Right: the $x$-axis is $\log(\log(n)/n)$ where $n$ is the number of points sampled on $\bigM_2$. The $y$-axis is the $\log$ of the estimated expectation of the bottleneck distance between the diagram obtained from an $\alpha$-filtration built on top of $n$ points sampled on $\bigM_2$ and $\dgm(\bigM_2)$.}
\label{fig:SphereConvergence}
\end{figure}

\begin{figure}
\centering
\includegraphics[width = 0.45\columnwidth]{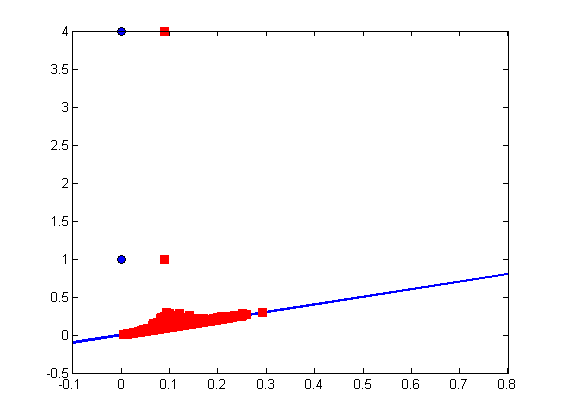}
\includegraphics[width = 0.45\columnwidth]{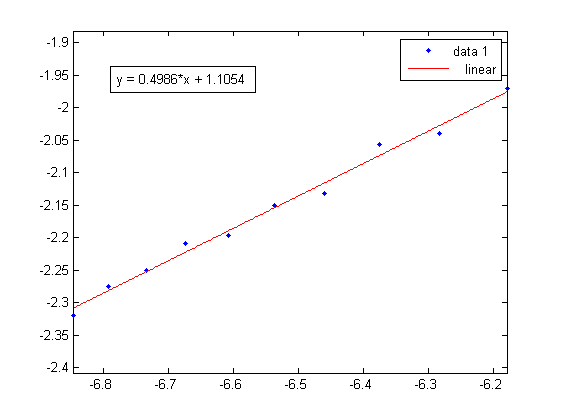}
\caption{Convergence rate for the persistence diagram of the $\alpha$-filtration built on top of points sampled on $\bigM_3$. Left: in blue the persistence diagram $\dgm(\bigM_3)$ of the distance to $\bigM_3$ ($1$-dimensional homology); in red a persistence diagram of the $\alpha$-filtration built on top of $n = 14000$ points randomly sampled on $\bigM_3$. Right: the $x$-axis is $\log(\log(n)/n)$ where $n$ is the number of points sampled on $\bigM_3$. The $y$-axis is the $\log$ of the estimated expectation of the bottleneck distance between the diagram obtain from  $\alpha$-filtration built on top of $n$ points sampled on $\bigM_3$ and $\dgm(\bigM_3)$.}
\label{fig:TorusConvergence}
\end{figure}

\begin{figure}
\centering
\includegraphics[width = 0.45\columnwidth]{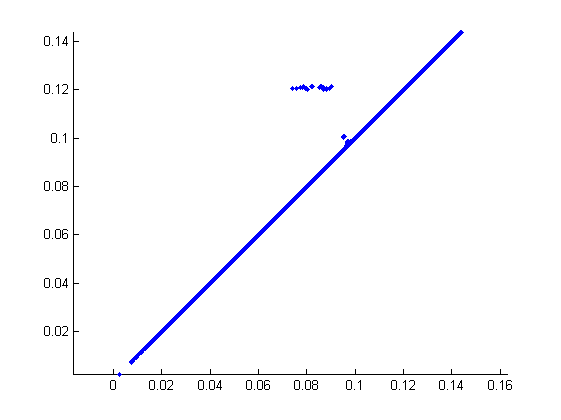}
\includegraphics[width = 0.45\columnwidth]{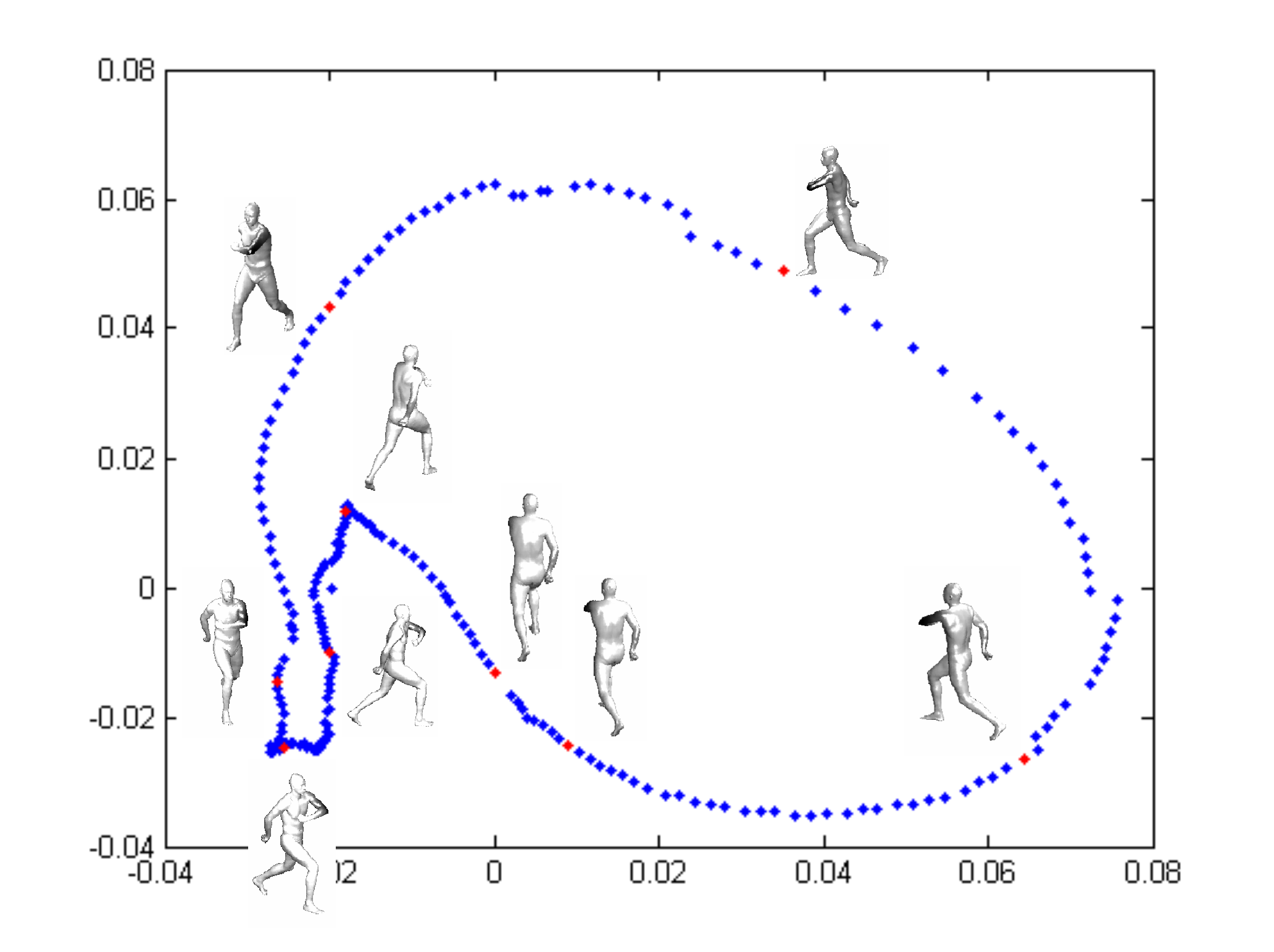}
\caption{Left: on the same figure the $1$-dimensional homology persistence diagrams of the Vietoris-Rips filtration of $20$ sets of $250$ points sampled on $\bigM_4$. Right: the plot of the embedding of  $\bigM_4$ in $\mathbb{R}^2$ using MDS.}
\label{fig:shaperotation}
\end{figure}

\section{Discussion and future works}
\label{sec:disc}

In previous works, the use of persistent homology in TDA has been mainly considered with a deterministic approach. As a consequence persistence diagrams were usually used as exploratory tools to analyze the topological structure of data. 
In this paper, we propose a rigorous framework to study the statistical properties of persistent homology and more precisely we give a general approach to study the rates of convergence for the estimation of persistence diagrams. The results we obtain open the door to a rigorous use of persistence diagrams in statistical framework. 
Our approach, consisting in reducing persistence diagram estimation to another more classical estimation problem (here support estimation) is based upon recently proven stability results in persistence theory that are very general. As a consequence, our approach can be adapted to other frameworks. For example, the estimation of persistence diagrams of functions (e.g. densities) is strongly connected to the problem of the approximation of such functions with respect to the sup norm. In particular, building on ideas developed in \cite{chazal2011geometric} and \cite{ccdm-dwmgi-11}, we intend to extend our results to persistence diagram estimation of distance-to-measure functions and topological inference from data corrupted by different kind of noise. 


In another direction, an interesting representation of persistence diagrams as elements of a Hilbert space has recently been proposed in \cite{Bubenik12}. Our results easily extend to this representation of persistence diagrams called {\em persistence landscapes}. 
Following this promising point of view, we also intend  to adapt classical kernel-based methods with kernels carrying topological information.

\appendix

\section{Lecam's Lemma}

The  version of Lecam's Lemma given below is from \cite{Yu97} (see also \cite{GPVW12}). Recall that  the total variation distance between two distributions $P_0$ and $P_1$ on a measured space $(\mathcal X, \mathcal B) $ is defined by 
$$ \TV(P_0,P_1)  = \sup_{B \in  \mathcal B} | P_0(B) - P_1(B) |.$$
Moreover, if $P_0$ and $P_1$ have densities $p_0$ and $p_1$
 for the same measure $\lambda$ on $\mathcal X$, then 
$$ \TV(P_0,P_1)    = \frac 1 2 \ell_1(p_0,p_1) := \int_{\mathcal X} |p_0-p_1|  d \lambda. $$

\begin{lemma} \label{Lem:Lecam} Let $\mathcal P $ be a set of distributions. For $P \in  \mathcal P$, let $\theta(P)$ take values in a metric space $(\X,\rho)$. Let
$P_0$ and $P_1$ in $\mathcal P$ be any pair of distributions. Let $X_1,\dots,X_n$ be drawn i.i.d. from some $P \in  \mathcal P$. Let $\hat 
\theta = \hat \theta(X_1,\dots,X_n) $ be any estimator of $\theta(P)$, then
\begin{equation*}
 \sup_{ P \in \mathcal P} \E _{P^n} \rho( \theta , \hat \theta  )  \geq   \frac 1 8    \rho \left( \theta(P_0), \theta(P_1) \right)   
\left[1 -  \TV(P_0,P_1) \right] ^{2 n } .
\end{equation*}
\end{lemma}

\section{Proofs}
\label{sec:Proofs}

\subsection{Proof of Theorem \ref{theo:ConvHaus}}

The proof follows the lines of the proof of Theorem 3 in \cite{CuevasRCasal04} . The only point to be checked is that the covering number of  $\X_\mu$
under the $(a,b)$-standard assumption can be controlled as when  $b=d \in \N$, the rest of the proof being unchanged.

The covering number $\cov(\X_\mu,r)$ of $\X_\mu$ is the minimum number of balls of radius $r$ that are necessary to cover $\X_\mu$:
$$ \cov(\X_\mu,r) = \min \left\{ k  \in \N^* \, : \, \exists  (x_1, \dots, x_k) \in (\X_\mu)^k  \textrm{ such that }\X_\mu = \bigcup_{i=1}^k B(X_i,r) \right\}.$$
The packing number $\pk(\X_\mu,r)$ is the maximum number of balls of radius $r$ that can be packed in $\X_\mu$ without overlap: 
$$\pk(\X_\mu,r) = \max \left\{ k  \in \N^* \, : \, \exists  (x_1, \dots, x_k) \in (\X_\mu)^k  \textrm{ such that } B(x_i,r) \subset \X_\mu \textrm{ and } \forall i \not = j \, B(x_i,r) \cap B(x_j,r) =
\emptyset \right\}$$
The covering and packing numbers are related by the following inequalities (see for instance \cite{Massart:07} p.71):
\begin{equation} \label{eq:pkcov}
 \pk(\X_\mu,2r) \leq \cov(\X_\mu,2r) \leq \pk(\X_\mu,r) .
\end{equation}
\begin{lemma} \label{lemma:packing-bound}
Assume that the probability $\mu$ satisfies a {\it standard} $(a,b)$-{\it assumption}. Then for  any $r>0$ we have 
\[
\pk(\X_\mu,r) \leq \frac{1}{a r^b}  \vee 1   \ \mbox{\rm and} \ \cov(\X_\mu,r) \leq \frac{2^b}{a r^b}  \vee 1 .
\] 
\end{lemma}

\begin{proof} The result is trivial for $r \geq a ^{-1/b}$. Let $r < a ^{-1/b}$ and let $p=\pk(\X_\mu,r)$, we choose a maximal packing $B_1 = B(x_1,r),  \cdots , B_{p} = B(x_{p},r)$ of $\X_\mu$.
Since the balls of the packing are pairwise disjoint and $\mu$ is a probability measure we have $\sum_{i=1}^p \mu(B_i) \leq 1$. Using that
$\mu(B_i) \geq a r^b$ we obtain that $p a r^b \leq \sum_{i=1}^p \mu(B_i)  \leq 1$ from which we get the upper bound on $\pk(\X_\mu,r)$. Since
from (\ref{eq:pkcov}) we have $\cov(\X_\mu,r) \leq \pk(\X_\mu,r/2)$ we immediately deduce the upper bound on $\cov(\X_\mu,r)$.
\end{proof}

\subsection{Proof of Proposition~\ref{prop:lowerboundAbs}}

\subsubsection*{Upper bound}
We first prove the upper bound. According to Corollary~\ref{cor:UpperRatesAbstract}, thanks to Fubini we have
\begin{eqnarray*}
\E  \left[ \bottle  (   \dgm (\Filt(\X_\mu)) ,   \dgm(\Filt(\widehat \X_n))  ) \right]  &\leq  & \int_{\e >0 } \P  \left[�  \bottle  (  
\dgm (\Filt(\X_\mu)) ,   \dgm(\Filt(\widehat \X_n))  ) > \e  \right]  d \e 
\end{eqnarray*}
Let $ \e_n  = 4 \left(  \frac {\log     n   } { a n} \right)^{1/b}$. By bounding the probability inside this integral by one on
$[0,\e_n]$,  we find that:
 \begin{eqnarray*}
\E  \left[ \bottle  (   \dgm (\Filt(\X_\mu)) ,   \dgm(\Filt(\widehat \X_n))  ) \right]  
&\leq  & \e _n + \int_{\e > \e _n  } \frac{8^b}{a}   \e^{-b} \exp(-na \e^b / 4 ^b )   d \e  \\
&\leq  & \e_n + \frac{4  n 2 ^b} {b} (na)^{-1/b}  \int_{u  \geq \log    n   } u ^{1/b-2}   \exp(- u  )   d u .
\end{eqnarray*}
Now, if $b \geq \frac 1 2$ then $u ^{1/b-2} \leq (\log n) ^{1/b-2}$ for any $u  \geq \log   n $  and then
 \begin{eqnarray}  \E  \left[ \bottle  (   \dgm (\Filt(\X_\mu)) ,   \dgm(\Filt(\widehat \X_n))  ) \right] &\leq & \e_n + 4 \frac {2^b} {b} \left(
\frac{\log n} n \right) ^{1/b} (\log n)^{-2}  \notag \\
&\leq & C_1(a,b)\left( \frac{\log n} n \right) ^{1/b} \label{majorEdb}
\end{eqnarray}
where the constant $C_1(a,b)$ only depends on $a$ and $b$. If  $0 < b < \frac 1 2$, let $p:=  \lfloor \frac 1 b \rfloor$ and then
 \begin{eqnarray*}
\int_{u  \geq u_n :=\log n}   u ^{1/b-2}\exp(- u  )   d u & = &  u_n ^{1/b-2} \exp(u_n) + (\frac 1 b -2) u_n^{1/b-3}\exp(u_n)  + \dots + \\
&   & + \prod_{i = 2}^p \left(\frac 1 b -i\right) u_n ^{1/b-p} \exp(u_n) +  \int_{u  \geq \log    n   } u ^{1/b-p - 1}   \exp(- u  )   d u \\
& \leq  & C_2(a,b) \frac{(\log n) ^{1/b-2}}{n}
 \end{eqnarray*}
where $C_2(a,b)$ only depends on $a$ and $b$. Thus (\ref{majorEdb}) is also satisfied for $  b < \frac 1 2$ and the upper bound is proved.

\subsubsection*{Lower bound}

To prove the lower bound, it will be sufficient to consider two Dirac distributions.  We take for $P_{0,n} = P_{x}$ the Dirac distribution on $\X_{0}
:= \{x\} $ and it is clear that $P_{0} \in \mathcal{P}(a,b,\bigM)$. Let $P_{1,n}$ be the distribution $\frac 1 n \delta_{x_n} + (1 -\frac 1 n) P_{0}$.
The support of $P_{1,n}$ is denoted $\X_{1,n} := \{x\}  \cup \{x_n\}$. Note that for any $n \geq 2$ and any $r \leq \rho(x,x_n)$:
$$ P_{1,n}\left(B(x,r) \right)=  1- \frac 1 n \geq  \frac  1 2   \geq \frac  1 {2 \rho(x,x_n)^b} r^b \geq a r^b  $$
and
$$ P_{1,n}\left(B(x_n,r) \right)=   \frac 1 n  = \frac  {  1 } {n \rho(x,x_n)^b} r^b \geq a r^b . $$
Moreover, for $r  > \rho(x,x_n)$, $ P_{1,n}\left(B(0,r) \right)=  P_{1,n}\left(B(x_n,r) \right) = 1$. Thus for any $r>0 $ and any $x \in \X_{1,n}$:
$$ P_{1,n}\left(B(x,r) \right) \geq a r^b  \wedge 1$$
and $P_{1,n}$ also belongs to $\mathcal{P}(a,b,\bigM)$.

The probability measure  $P_{0} $ is absolutely continuous with respect to $P_{1,n}$ and the density of $P_{0}$ with respect to $P_{1,n}$  is $
p_{0,n} :=   \frac n {n - 1} \one_{\{x\}}$.  
Then 
\begin{eqnarray*}
TV(P_{0},P_{1,n}) & = & \int_{\bigM} | 1 -  \frac n {n - 1} \one_{\{x\}}  |� \; d P_{1,n}  \\
& = &  \frac 2 {n} .
\end{eqnarray*}
Next, $ \left[ 1 - TV(P_{0},P_{1,n}) \right] ^{2 n }  = (1-\frac 2 n ) ^{2n}  \rightarrow e^{-4}$  as $n$ tends to infinity. It remains to compute
$\bottle(  \dgm(\Filt(\X_{0})), \dgm(\Filt(\X_{1,n})) )$. We only consider here the Rips case, the other filtrations can be treated in a similar way.
The bar code of $\Filt(\X_{0})$ is composed of only one segment $(0, +\infty)$ for the 0-cycles. The barcode of  $\Filt(\X_{1,n})$ is composed of the
segment of $\Filt(\X_{0})$ and one more 0-cycle : $(0,\rho(x,x_n))$. Thus we have:
\begin{eqnarray*}
 \bottle(  \dgm(\Filt(\X_{0})), \dgm(\Filt(\X_{1,n})) ) &=& d_\infty \left( \Delta ,   (0,\rho(x,x_n)) \right) \\
  &= &   \frac {\rho(x,x_n)} 2  .
\end{eqnarray*}
The proof is then complete using  Lecam's Lemma (Lemma \ref{Lem:Lecam}).

\subsection{Proofs for Section \ref{sub:OptSmoothBoundaries}} \label{proofSec41}

\begin{lemma} \label{lem:dhG0front}
\begin{enumerate}
\item Under assumption $[B]$, we have 
$ \dhaus(G_0,  \X_{f}) = 0$.
\item Under Assumptions  $[A]$ and $[B]$, $\mu$ satisfies a standard assumption with $b=\alpha + d$ and with  $a$ depending on $\mathcal F(\alpha)$.
\end{enumerate}
\end{lemma}
\begin{proof}
First, note that we always have
\begin{equation} \label{eq:G0Xmu}
 \overset{\circ}{G_0}  \subset  \X_{f} \subset \widebar{G_0}.
\end{equation}
Indeed, if $ \overset{\circ}{G_0} \cap ( \chi \setminus \X_{f}) $ is non empty, let $x$ be in the intersection. Then there exists $\varepsilon >0$
such that $B(x,\varepsilon) \subset G_0$ and $B(x,\varepsilon) \subset ( \chi \setminus \X_{f})$ since $\X_{f}$ is assumed to be closed. The
first inclusion then gives that $\mu(B(x,\varepsilon)) > 0 $ whereas the second inclusion gives that $\mu(B(x,\varepsilon)) = 0 $. Thus  $
\overset{\circ}{G_0} \cap ( \chi \setminus \X_{f}) $ is empty, the second inclusion in (\ref{eq:G0Xmu}) is obvious since $\X_{f}$ is assumed to
be closed. 

Then, 
\begin{eqnarray}
 \dhaus(\X_{f},G_0) & =&  \max ( \sup_{x \in \X_{f}} d(x,G_0) , \sup_{x \in G_0} d(x,\X_{f}) ) \notag \\
 & =&  \max ( \sup_{x \in \X_{f}} d(x, \widebar{G_0}) , \sup_{x \in \widebar{G_0}} d(x,\X_{f}) ) \notag \\
& =&   \sup_{x \in \widebar{G_0}} d(x,\X_{f}) \notag \\
& =&   \sup_{x \in \partial G_0} d(x,\X_{f}) \label{partialG0Xmu}
\end{eqnarray}
where we use the continuity of the distance function for the second equality and (\ref{eq:G0Xmu}) for the two last ones. It follows from assumption
$[B]$ that for any $x \in \partial G_0$, $d(x,\overset{\circ}{G_0} ) = 0$. Thus $d(x,\X_{f} ) = 0$ according to (\ref{eq:G0Xmu}) and we have proved
that (\ref{partialG0Xmu}) is equal to zero.

We now prove the second point of the Lemma. Let $x \in \bar{G_0}$ and let $r >0 $ such that 
\begin{equation}
\label{eq:rlessr0}
 \frac r 2  \left(1 \wedge \frac 1 {C_b}  \right)   <    \varepsilon_0  \wedge \left( \frac {\delta_a}{ C_a} \right) ^{1/ \alpha}   .
\end{equation}
 According to Assumption
$[B]$, for $\varepsilon = \frac r 2  \left(1 \wedge \frac 1 {C_b}  \right)$, there exists $y \in I_{\varepsilon} (G_0)$ such that $d(x,y) \leq \ C_b
\varepsilon 
\leq \frac r 2 $. Then, there exists $z \in I_{\varepsilon}$ such that $y \in B(z,\varepsilon)  \subset I_{\varepsilon}$. Since $ \varepsilon \leq
\frac r 2  $ we find that $B(z,\varepsilon) \subset B(x,r) \cap G_0 $. Thus,
\begin{eqnarray*}
 \mu \left(B(x,r) \right)  & \geq  &  \int_{B(z,\varepsilon)} f(u) \ d \lambda (u)  \\
   & \geq  &  \int_{B(z,\varepsilon)} \delta_a  \wedge C_a d(u, \partial G_0 )^\alpha  \ d \lambda (u) \\
& \geq  & C_a  \int_{B(z,\varepsilon)} \left( \varepsilon - \|u - z\| \right)^\alpha \ d \lambda (u) \\
& \geq  & C_a s_{d-1}  \int_{0}^\varepsilon  \left( \varepsilon - r \right)^\alpha r ^{d-1} \ d r   
\end{eqnarray*}
where $s_{d-1}$ denotes the surface area of the unit $d-1$-sphere of $\R^d$, and where we have used Assumption $[A]$ for the second inequality and the
fact $C_a \varepsilon^\alpha \leq \delta_a$ for the third one. Finally we find
that for any $r$ satisfying (\ref{eq:rlessr0}):
\begin{eqnarray*}
\mu \left(B(x,r) \right) & \geq  & \frac{ C_a s_{d-1} (d-1) !  }{(\alpha+ 1) \dots(\alpha + d) } \varepsilon^{\alpha + d} \\
 & \geq  &  \frac{ C_a s_{d-1} (d-1) !  (1 \wedge \frac 1 {C_b})^{\alpha + d}}{ 2 ^{\alpha +d}(\alpha+ 1) \dots(\alpha + d) } 
  r^{\alpha +d}
 \end{eqnarray*}
and we obtain that $\mu$ satisfies that standard assumption with $b=\alpha + d$.
\end{proof}

\subsubsection*{Proof of Proposition ~\ref{prop:dplusalpha}}

The first point of the proposition is an immediate consequence of Theorem 3 in \cite{SinghScottNowak09} and Lemma~\ref{lem:dhG0front}. We now prove
the lower bound by adapting some ideas from the proof of Proposition 3 in \cite{SinghScottNowak09} about the Hausdorff lower bound. At the price
of loosing a logarithm term in the lower bound, we propose here a proof based on a two-alternative analysis.

The function $f_0$ is defined on $ \chi$ as follows for  $r_0 > 0$ small enough:
 $$f_0 = \left\{ 
    \begin{array}{ll}
        C_a \|x\| ^\alpha & \mbox{if }  \| x  \|  \leq r_0  \\
        C_ 0 & \mbox{if } r_0   \leq  \| x  \|  \leq 2 r_0 \\ 
	C_a (3 r_0 - \|x\|) ^\alpha & \mbox{if } 2 r_0 \leq  \| x  \|  \leq 3 r_0 \\
	0 & \mbox{elsewhere }
    \end{array}
\right.
$$
where 
$$C_0 =   \frac{  1 - C_a s_{d-1} r_0 ^ {d + \alpha} (\frac 1 {d+  \alpha} + I _\alpha) }{   s_{d-1} r_0 ^ {d} (2 ^d - 1 ) / d } 
\hskip 1cm \textrm{ with } I _\alpha=  \int_{2  } ^{3  } d ^{d-1} (3 - u ) ^{\alpha}  d u . $$ 
For $n \geq 1$  let $\varepsilon _n := n ^{-1 / (d + \alpha)}$, the function $f_{1,n}$ is defined on $ \chi$ by 
 $$f_{1,n} = \left\{ 
    \begin{array}{ll}
	 \|x\| ^\alpha & \mbox{if }  \varepsilon_n \leq  \| x  \|  \leq r_0 \\
        C_ {1,n} & \mbox{if } r_0   \leq  \| x  \|  \leq 2 r_0 \\ 
	C_a (3 r_0 - \|x\|) ^\alpha & \mbox{if } 2 r_0 \leq  \| x  \|  \leq 3 r_0 \\
	0 & \mbox{elsewhere }
    \end{array}
\right.
$$
where 
\begin{eqnarray*}
C_ {1,n} & =&    \frac{  1 - C_a s_{d-1}  \left\{ r_0 ^ {d + \alpha} (\frac 1 {d+   \alpha} + I _\alpha) - \frac{ \varepsilon_n
^{d+\alpha}}{d+\alpha} \right\} }{  s_{d-1} r_0 ^ {d} (2 ^d - 1 ) / d} \\
& =&  C_0 + \frac{d C_a  \varepsilon_n ^{d+\alpha}  }{(d+\alpha) r_0 ^ {d} (2 ^d - 1 ) }.
\end{eqnarray*}
We assume that $\delta_a$ is small enough so that we can choose $r_0$ such that $\delta_a \leq  C_0$  for $n$ large enough. Then $f_{0}$ and $f_{1,n}$
are both densities and they both belong to $\mathcal F (\alpha)$ for $n$ large enough. The support of  $f_{0} d \lambda
$ is equal to $ \X_{0} := \bar B (0, 3 r_0)$ whereas the support of  $f_{1,n} d \lambda $ is equal to $ \X_{1,n} = \bar B (0, 3 r_0) \setminus
\bar B (0, \varepsilon_n)$. Next,
\begin{eqnarray*}
\TV ( f_0 \, d \lambda, f_{1,n} \, d \lambda )  & = &  \int_{\chi} |f_{0}   - f_{1,n} |  d x  \\
& = &  s_{d-1} C_a \int_{0}^{\varepsilon_n }  r  ^ {\alpha + d-1} \, d r  + s_{d-1} \int_{r_0} ^ {2 r_0} ( C_{1,n} - C_{0}) r  ^ {d-1}  d r 
\\
& = & \frac{ 2  s_{d-1}C_a }{d+\alpha}  \varepsilon_n ^{d+\alpha}  
\end{eqnarray*}
Note that $(1 - \TV ( f_0 \, d \lambda, f_{1,n} \, d \lambda ) ]^{2 n } \rightarrow \exp( -  \frac{4  s_{d-1}C_a }{d+\alpha})$ as $n$ tends to
infinity. It remains to compute $\bottle(  \dgm(\Filt(\X_{0})), \dgm(\Filt(\X_{1,n})) )$. We only consider here the Rips case, the
other filtrations can be treated in a similar way. The bar code of $\Filt(\X_{0})$ is composed of only one segment $(0, +\infty)$ for the 0-cycles.
The barcode of  $\Filt(\X_{1,n})$ is composed of the segment of $\Filt(\X_{0})$ and one more 1-cycle : $(0,2 \varepsilon_n)$. Thus we have:
\begin{eqnarray*}
 \bottle(  \dgm(\Filt(\X_{0})), \dgm(\Filt(\X_{1,n})) ) &=& d_\infty \left( \Delta ,   (0,\varepsilon) \right) \\
  &= &    \varepsilon _n .
\end{eqnarray*}
We then finish the proof using  Lecam's Lemma.

\subsection{Proof of Proposition ~\ref{prop:GPVW}}

We only need to prove the lower bound since the upper bound is a direct corollary of Theorem~3 in  \cite{GenoveseEtAl2012}.  To prove the lower bound,
we may use the particular manifolds defined in \cite{GPVW12} and also used by the same authors for the proof of Theorem~2 in \cite{GenoveseEtAl2012}.
Without loss of generality, we assume that $\chi = [-L,L] ^D$ and that $\kappa <   L  / 2 $. For $ \ell \leq L$, let $M$ and  $M'$ be the two
manifolds of $\chi$ defined by 
 $$ M = [-\ell, \ell]^D \cap  \{ x \in \chi  \,  |� \,  x_{d+1} = \dots = x_{D} = 0 \}  \quad \textrm{ and } \quad M' =   2 \kappa e_{d+1}  + M $$
 where $e_{d+1}$ is the $d+1$-th vector of the canonical basis in $\R^D$. We  assume that $\ell$ is chosen so that $ b < 2 (2 \ell)^{-d}  <  B $. Let
$\mu_0$ be the uniform measure on $\X_0 := M \cup M'$ and then $\mu_0 \in \mathcal H $.

According to Theorem  6 in \cite{GPVW12}, for $0 <  \gamma < \kappa$, we can define a manifold $M_\gamma$ which can be seen as a perturbation of $M$ such that:
\begin{itemize}
\item  $\Delta (M_\gamma) = \kappa$
\item  $\dhaus(M_\gamma,M) = \gamma $ and $\dhaus(M_\gamma,M') = 2 \kappa - \gamma $
\item  If $A = \{ x \in M_\gamma \, | \, x \notin M\}$ then $\mu_1 ( A) \leq C \gamma^{d/2}$ where $C >0$ and where  $\mu_1 $ is the uniform measure on $\X_1:=M_\gamma \cup M'$. 
\end{itemize}
For small enough $\gamma$ we see that $\mu_1$ satisfies $[H_2]$ and thus $\mu_1 \in \mathcal H $. 

As before, we only consider here filtrations of Rips complexes. The persistence diagrams of $\Filt(\X_0)$  and $\Filt(\X_1)$ are exactly the same except for the diagram of 0-cycles : the first filtration has a bar code with a segment $(0,2 \kappa)$ whereas the corresponding bar code for $\Filt(\X_1)$ is $(0,2 \kappa - \gamma)$. Thus, $\bottle( \Filt(\X_0), \Filt(\X_1)) = \gamma$. Moreover, $\TV(\mu_0,\mu_1) \leq |�\mu_0(A)  -  \mu_{1} (A) |  \leq C \gamma^{d/2}$. Finally, we choose $\gamma = (1/n) ^{d/2}$  as in the proof of Theorem~2 in \cite{GenoveseEtAl2012} and we conclude using Lecam's Lemma.

\section*{Acknowledgements}
The authors acknowledge the support of the European project CG-Learning EC contract No. 255827, and of the ANR project GIGA (ANR-09-BLAN-0331-01).
\bibliographystyle{plain}
\bibliography{MetricBiblio}

\end{document}